%% file: pless14.tex
\documentclass[11pt,reqno]{amsart}
\usepackage{amsmath,amsfonts,amssymb,amsthm,amsopn,cite,mathrsfs}
\usepackage{epsfig,verbatim}
\usepackage{subfigure}
\usepackage{xcolor}

\include{macros}
\newcommand{\ep}{\varepsilon}
\newcommand{\cdo}{\, \cdot \,}
\newcommand{\Rdst}{{\mathbb{R}^d}}

\newcommand{\s}{\mathcal{S}(\Rdst)}

\newcommand{\Rtdst}{{\mathbb{R}^{2d}}}

\newcommand{\Zdst}{{\mathbb{Z}^d}}
\newcommand{\Ztdst}{{\mathbb{Z}^{2d}}}

\newcommand{\menv}{\triangle}
\newcommand{\rel}{\mathop{\mathrm{rel}}}

\newcommand{\nm}[2]{\Vert #1\Vert _{#2}}

\newcommand{\set}[2]{\big\{ \, #1 \, :  \, #2 \, \big\}}

\def\supp{\operatorname{supp}}
   
\newcommand{\bignorm}[1]{\bigl\lVert#1\bigr\rVert}

\newcommand{\scal}[2]{\langle #1,#2\rangle}
\newcommand{\Schur}{\operatorname{Schur}}

\begin{document}
\begin{abstract}
We extend the stability and spectral invariance of convolution-dominated matrices to
the case of quasi-Banach algebras $p<1$. As an application we
construct a spectrally invariant quasi-Banach algebra of
pseudodifferential operators with non-smooth symbols that generalize
Sj\"ostrand's results.   
\end{abstract}

\title[Spectral Invariance of Quasi-Banach Algebras of Matrices and
PDOs.]{Spectral Invariance of Quasi-Banach Algebras of Matrices and
  Pseudodifferential Operators} 
\author{Karlheinz Gr\"ochenig}
\address{Faculty of Mathematics \\
University of Vienna \\
Oskar-Morgenstern-Platz 1 \\
A-1090 Vienna, Austria}
\email{karlheinz.groechenig@univie.ac.at}
\subjclass[2010]{}
\author{Christine Pfeuffer}
\address{Faculty of Mathematics\\
Martin-Luther-Universit\"at Halle-Wittenberg\\
06099 Halle (Saale), Germany}
\email{christine.pfeuffer@mathematik.uni-halle.de}
\author{Joachim Toft}
\address{Faculty of Mathematics\\
Lineaus University\\
Universitetsplatsen 1\\
35195 Växjö, Sweden
}
\email{joachim.toft@lnu.se}
\date{}
\maketitle

\section{Introduction}

Spectral invariance is an important phenomenon 
for applications in the field of partial differential equations and  in the theory of \psdo
s. The first result due
to Beals~\cite{beals77} asserts that the inverse of a \psdo  \ $T$ that is (i) invertible on $\lrd
$ and (ii) with a symbol in the  H\"ormander class $S^0_{0,0}$  is again a \psdo\ with
a symbol in the same class. In other words, this class of \psdo s is
inverse-closed (closed under inversion)
in the algebra  $\cA $ of
\psdo s with $S^0_{0,0}$-symbols.
 As a consequence
the spectrum of $T$ 
is independent of 
the weighted $L^p(\mathbb{R}^d)$ space or of the choice of $B^{s,a}_{p,q}(\mathbb{R}^d)$, see \cite{LeopoldTriebel, Schrohe1990}. 
 This phenomenon is often referred to as  spectral
invariance, and the resemblance to Wiener's lemma for absolutely
convergent Fourier series has also motivated the terminology of $\cA $
being a Wiener algebra.

The next important step in the theory of spectral invariance of \psdo
s was made by Sj\"ostrand~\cite{Sjo95} who introduced a class of non-smooth
symbols for which the associated algebra of \psdo s is spectrally
invariant in $\cB (\lrd )$. This class, nowadays called the
Sj\"ostrand class, turned out to be an already  known function space 
that is paramount in \tfa , namely the \modsp\ $M^{\infty , 1} (\rdd
)$. This connection spawned an intensive investigation of \psdo s with
\tf\ methods ~\cite{Gro06,GR08,GS07,GrHe99,FeGaTo2014,Sun2007,Sun2011,To01,To17,Sjo95,Te06}. Among the new
results obtained by \tf\ methods was, firstly,  a characterization of the
H\"ormander class and the Sj\"ostrand class by means of the matrix
associated to a \psdo\ with respect to a Gabor frame. Secondly, \tfa\
established a firm connection between the off-diagonal decay of these
matrices and the corresponding properties (boundedness, algebra,
inverse-closedness) of \psdo s. In this way,  every (solid) inverse-closed
subalgebra of $\cB (\ell ^2(\zdd ))$  can be mapped to an algebra of
\psdo s that is inverse-closed in $\cB (\lrd )$~\cite{GR08}. In
contrast to the classical hard analysis methods the \tf\ approach is
so flexible that the theory can even be formulated for \psdo s on
locally compact abelian groups~\cite{GS07}. 

The goal of this  paper is the extension of the theory of  spectrally
invariant algebras  of pseudodifferential operators to the realm of
quasi-Banach algebras. Quasi-Banach algebras are interesting in their
own right, but quasi-Banach spaces and associated operators occur naturally in approximation theory and data 
compression problems, see e.g. \cite{devoretemlyakov96,GS20}.
Additionally   in \tfa\ they occur in the
formulation of \up s~\cite{GaGr02}.

To formulate our main results, we briefly recall the definition of
\modsp s and the Weyl form of \psdo s.
For a fixed non-zero Schwartz function  $g \in \s$ and a
tempered distribution $f$ the
\textit{short-time Fourier transform} $V_{g} f$ is the function on
$\mathbb{R}^{2d}$ defined by the formula
\begin{align}\label{def:STFTIntro}
	V_{g}f(x,\xi) := \int_{\mathbb{R}^d} f(t) \overline{ g
  (t -x)} e^{-2\pi i t \cdot \xi } dt \, 
\end{align}
(with suitable interpretation of the integral). For 
 $0<p,q \leq \infty$ the  (unweighted)  \emph{modulation space}
 $M^{p,q}(\rd)$ is defined by the quasi-norm 
\begin{align}
	\|f\|_{M^{p,q}}:=
	\left( \int_{\rd} \left( \int_{\rd} |V_g f(x, \xi )|^p\,  dx 
	\right)^{q/p} d\xi \right)^{1/q}
\label{Eq:ModNorm}
\end{align}
(with usual modifications in case $p=\infty$ or $q=\infty$)
and consists of all tempered distributions $f$ with finite
quasi-norm. 
Modulation spaces on $\rdd $ serve as symbol classes for \psdo s. 
Our focus will be on the symbol class $M^{\infty ,p_0}(\rdd )$
for $p_0<1$. This is only  a quasi-Banach space.

Given a symbol $\mathfrak{a}$ on $\rdd $, the corresponding
\psdo\ in the Weyl calculus is defined formally as
$$
\mathfrak{a}^w f(x) = \intrd \intrd
\mathfrak{a}\big(\frac{x+y}{2},  \xi) f(y)
	e^{2\pi i(x-y)\cdot \xi}\, dy d\xi \, ,
$$
(again with a suitable interpretation of the integral and  for $f\in
\cS (\rd )$).

With these definitions our main results can be stated as follows. 
\begin{tm}[Spectral invariance]
\label{tm:SpectralInvarianceNew}
Let $p_0\in (0,1]$ and $\mathfrak{a} \in M^{\infty,p_0 }(\rdd )$
be such that $\mathfrak{a}   ^w$ is invertible on
$M^p(\rd)$ for some $p \in [p_0,\infty]$.
Then $(\mathfrak{a} ^w)\inv =\mathfrak{b} ^w $ for some
$\mathfrak{b} \in M^{\infty,p_0 }(\rdd )$. 
\end{tm}

As a consequence we obtain that the spectrum of a \psdo{}  with a
symbol in $M^{\infty ,p_0}(\rdd )$ is independent of the space on
which it acts. 
\begin{tm}[Spectral invariance on Modulation Spaces] \label{tm:SpectralInvarianceModSpaces}
	If  $\mathfrak{a}  \in M^{\infty,p_0 }(\rdd )$ for $p_0\in
        (0,1]$ and  $\mathfrak{a}    ^w$ is invertible on $M^p(\rd)$
        for some $p \in [p_0,\infty]$, then $\mathfrak{a} ^w$ is also
        invertible on  $M^{q}(\rd)$ for all $q \in [p_0,\infty)$. 
\end{tm}
More generally, we show in Theorem \ref{thm:SprectralInvariance} the invertibility of $\mathfrak{a} ^w$ on the more broad class of modulation spaces $M^{r,q}(\rd)$ with $r, q \in [p_0,\infty)$ under the assumptions of the previous theorem.\\

\textbf{Methods.} We follow the proof outline of \cite{GR08}. The
first step  is to study the matrix representation of a \psdo\ with respect
to a Gabor frame and then derive a characterization of the symbol
class in terms of the off-diagonal decay of the associated
matrix.

In the second step  this leads to the study of spectrally invariant matrix
algebras. The appropriate class in our context are
convolution-dominated matrices, i.e., matrices $A=(a_{\lambda ,\rho
  })_{\lambda,\rho \in \Lambda }$ with an off-diagonal
  decay of the  form
$$
|a_{\lambda ,\rho }| \leq H(\lambda -\rho )
$$
for a (smooth) function $H$ in $L^p(\rd)$. It turns out that such  matrix
classes are spectrally invariant in $\cB (\ell ^p(\Lambda ))$. 
To offer a glimpse of this aspect, we formulate a very special case of
our main result on matrices that does not require technical details.

\begin{tm}
  Let $A=(a_{kl})_{k,l\in \zd } $ be a matrix over the index set $\zd
  $. Assume that there exists a sequence $h\in \ell ^{p_0}(\zd )$ for
  $0<p_0\leq 1$, such that
  $$
  |a_{kl}| \leq h(k-l) \qquad \text{ for all } k,l \in \zd \, .
  $$

  (i) Spectral invariance:  If  $A$ is invertible on some $\ell ^p(\zd )$ for $p\in
  [p_0,\infty ]$, then $A$ is invertible on all $\ell ^q (\zd
  )$ for $q\in   [p_0,\infty ]$.

  (ii) Spectral Stability: If for some $\ell ^p(\zd )$ with $p\in
  [p_0,\infty ]$, $A$ satisfies the stability condition
  $$
  \|Ac\|_p \geq C \|c\|_p \text{ for all } c\in \ell ^p(\zd )\, ,
  $$
  then $A$ satisfies $  \|Ac\|_q \geq C_q \|c\|_q \text{ for all } c\in \ell ^q(\zd )$ with $q\in [p_0,\infty ]$. 
\end{tm}

In our approach  we follow Sj\"ostrand ingenious proof of
Wiener's lemma for absolutely convergent Fourier series~\cite{Sjo95} and built on
the presentation in~\cite{GOR15}.  Our ultimate results on spectral
invariance and stability (Theorems~\ref{Thm:LowerBound} and
\ref{tm:SpectralInvariance})  are a significant extension
of the above preliminary statement and provide  several new
facts of spectral invariance: 
\begin{itemize}
    \item[(i)]  they  yield both spectral stability and
spectral invariance, 
    \item[(ii)] they are formulated  with respect to arbitrary operator algebras $\cB
(\ell ^p)$ (not just $\cB (\ell ^2)$ as is usually done) and \item[(iii)] they cover the general case of
quasi-Banach algebras, and 
\item[(iv)] in addition we treat arbitrary index sets and not just $\zd $ or a discrete abelian group as in most  references. 
\end{itemize}


Technically, the study of quasi-Banach algebras of
convolution-dominated matrices is the main part of our paper, its
application to \psdo s is then  based on our
analysis in \cite{GR08}. Our proof contains some new features and
avoids  the functional
calculus associated with the pseudo-inverse. These arguments may be
useful in other contexts as well. 

\vspace{3mm}
\textbf{Related results.}
There
are numerous results on the spectral invariance of matrices, we mention~\cite{ABK2008,Ba90,Ba97,Sun2005,Sun2007,Sun2011,Shin2016,Sjo95,Jaffard1990} for a small
sample and ~\cite{Gro10} for a survey. As long as the index set is a
discrete abelian group, one can use
methods from harmonic analysis to 
establish spectral invariance.
This line of thought goes back to
Bochner and Philipps and is used
in~\cite{Ba90,Ba97,DaSo22_2,CoGi22} and many others. All these
proofs break down, however, when 
unstructured index sets are considered.

The extension of spectral invariance to 
quasi-Banach algebras  of matrices and operators over $\zd $  is  the subject of
the  recent papers~\cite{CoGi22,DaSo22_2}. While there is some
thematic overlap, not all results are directly comparable. On the one
hand,  we
restrict our attention to unweighted $\ell ^p$-spaces, whereas
\cite{CoGi22,DaSo22_2} include weights.  On the other hand, these papers prove that
invertibility of convolution-dominated matrices on the Hilbert space
$\ell ^2$ implies invertibility  on all $\ell ^p$ for $p> p_0$. Our
results also provide a converse, namely that invertilibity on some
$\ell ^p$ implies invertibility on $\ell ^2$. 

As for the application to pseudodifferential operators, 
the Banach algebra case $p_0=1$ for the symbols with invertibility on  $M^{2,2} = L^2(\rd )$ was already proved by
Sj\"ostrand~\cite{Sjo95}. The quasi-Banach algebra case $p_0<1$ with
invertibility on  $M^{2,2} = L^2(\rd )$ was established in~
\cite{CoGi22} based on the  \tf\ methods developed in~\cite{GR08}. We
want to emphasize that the general case where we start with
invertibility on $M^p(\rd ), p \neq 2$, is much more involved, as there are no
Hilbert space techniques available. For self-adjoint \psdo s
$(\mathfrak{a}^w)^*= \mathfrak{a}^w$ one could argue with  duality and
interpolation to reduce to the case of  invertibility on $L^2(\rd )$, but there is
no cheap trick for non-self-adjoint \psdo s. When we start with
invertibility on a quasi-Banach  space $M^p(\rd ), p<1$, even duality
is no longer useful. For this  reason we disregarded the methods of
Bochner-Philipps as used in~\cite{DaSo22_2,GR08,CoGi22}  for the construction
of spectrally invariant matrix algebras and  instead  exploit
Sj\"ostrand's original proof of Wiener's lemma in ~\cite{Sjo95}. 

\vspace{3mm}

The paper is organized as follows: In Section~2 we collect the
relevant definitions about sequence spaces and amalgam spaces. In
Section~3 we treat the spectral invariance of convolution-dominated
matrices. We treat both the stability of such matrices on the
quasi-Banach spaces $\ell ^p, p<1$ and the spectral invariance of the
algebra of convolution-dominated matrices. The main results are
Theorems~\ref{Thm:LowerBound} and ~\ref{tm:SpectralInvariance}. In
Section~4 we first recapitulate the definitions of modulation spaces
and Gabor frames and various calculi of pseudodifferential operators
and then prove our main theorems that are already stated in the
introduction. For completeness we have postponed some easy and known
proofs to the appendix.

\section{Preliminaries} \label{sec:Prelim}

%

For the convenience of the reader we now list up 
the definitions of some function spaces and their 
properties needed during this paper. We start with 
the sequence spaces. 

\subsection{Sequence space $\ell ^p$}

For each $0 < p \leq \infty$ and discrete
set $J$, we recall that
the set $\ell ^p(J)$ consists of all
complex-valued sequences $a$
such that
\begin{align*}
\|a\|_{p}:=\left( \sum _{j \in J} 
|a_j|^p \right)^{1/p},
\qquad
a=(a_j)_{j\in J}
\end{align*}
is finite (with usual modifications for 
$p=\infty$). Then $\ell ^p(J)$ is a  quasi-Banach space  with quasi-norm
$\|\, \cdot \, \| _p$, which   is even a norm
if $p\ge 1$.


We recall the following properties for 
$\ell^p(J)$.

\begin{lemma} \label{plone}
Suppose $0<p\leq 1$, $J$ is a discrete set,
$\Lambda$ is  countable. 
Let $a=(a_j)_{j\in J}\in \ell ^p(J)$
and $b,b_\lambda \in \ell ^p(J)$,
$\lambda \in \Lambda$.
Then the following is true:
\begin{enumerate}
\item[(i)] $\|a\|_1^p \leq \| a\|_p^p$, or 
$|\sum_{j \in J} a_j|^p \leq 
\sum _{j \in J}
|a_j|^p$;

\vspace{0.2cm}

\item[(ii)] if $\sum _{\lambda \in \Lambda} \|b_{\lambda}\|^p_p<\infty$, then
$\sum _{\lambda \in \Lambda} b_{\lambda}$
is uniquely defined as an element in
$\ell ^p(J)$, and
$\| \sum _{\lambda \in \Lambda}
b_{\lambda} \| _p^p
\leq \sum _{\lambda \in \Lambda} 
\|b_{\lambda}\|_p^p$;

\vspace{0.2cm}

\item[(iii)] If $J=\zd $, then also  $\|a \ast b \|_p \leq \|a\|_p \, \|b\|_p$. 
\end{enumerate}
\end{lemma}

\subsection{Wiener Amalgam Space}

Let $X=L^{\infty}(\rd)$ or $X=C_b(\rd):=
L^{\infty}(\rd)
\cap C(\rd)$, and let $K \subseteq \rd$ be
convex and compact with positive volume.
The Wiener Amalgam space
$W(X, L^{p_0})$ for $0<p_0\leq \infty$
consists of all $f \in  X$ such that
\begin{align*}
\bignorm{f}_{K,W(X,L^{p_0})}:=
\left( \int_{\mathbb{R}^d} \|f\|^{p_0}_{L^{\infty}(x+K)}
\, dx
\right)^{1/p_0} < \infty \, .
\end{align*}

\par

This is always a quasi-norm, and a norm, if $p_0\ge 1$ (see e.g.
\cite{He03}). For  $p_0=\infty$ we have $W(X,L^\infty ) = X$. 
By compactness it follows that $W(X, L^{p_0})$ is independent
of the choice of $K$, and different $K$ yield equivalent
quasi-norms. For convenience we set
$\bignorm {\cdo}_{W(X,L^{p_0})}
=\bignorm {\cdo}_{B_1(0),W(X,L^{p_0})}$.

\par

\begin{rem}\label{Rem:InvarWienerAmalg}
We observe that every continuous function with compact support is
contained in $W(C_b, L^p)$ for all $p>0$ and that $W(C_b,L^p)$
is translation  invariant
(see e.{\,}g. \cite{CorNic1} and the
references therein).
\end{rem}


The  Wiener Amalgam space 
$W(C_b, L^{p_0})$, $0<p_0<\infty$ arises naturally in the formulation
of sampling inequalities. 
We first recall that a  set $\Lambda \subseteq \rd$
is  \emph{relatively separated}, if
\begin{align}
\rel(\Lambda):=\sup \{ \# (\Lambda \cap B_1(x)):
x \in \rd \} < \infty.
\end{align} 

\begin{lemma}
\label{ConnWienAmalgSpaceSequenceSpace2}
Let $0<p_0<\infty$, $\Lambda \subseteq \rd$
be relatively separated and let
$H \in W(C_b, L^{p_0})$.
Then $(H(\lambda))_{\lambda\in \Lambda}
\in \ell^{p_0}(\Lambda)$.
\end{lemma}

Lemma \ref{ConnWienAmalgSpaceSequenceSpace2} 
follows by straight-forward estimates. In order
to be self-contained, a proof of the result
is given in Appendix \ref{App:A}.

\section{Spectral Invariance of Convolution-Dominated
  Matrices} \label{sec:SpectralInvariance} 

In this section we prove  a spectral invariance result for infinite 
dimensional convolution-dominated matrices. For 
this we first list  some needed auxiliary tools. 

We always  denote the conjugate
exponent of $p\in [1,\infty ]$
by  $p'= \tfrac{p}{p-1}$, so that  $p'\in [1,\infty ]$
and
$$
\frac 1p + \frac 1{p'}=1.
$$
We identify a matrix $A=(a_{\lambda,\rho })_{\lambda \in \Lambda ,
  \rho \in \Pi }$ indexed by $\Lambda $ and $\Pi$ with a linear operator
\begin{align*}
	(A b)_{\lambda} := \sum_{\rho \in \Pi} a_{\lambda
  \rho} b_{\rho } \qquad \text{ for  } b=(b_{\rho})_{\rho
  \in \Pi} \, .
\end{align*}
Then $A$ is always well-defined on finite sequences and $A$ maps
finite sequences on $\Pi$ to arbitrary sequences on $\Lambda
$. Some  boundedness properties of matrices on $\ell ^p$-spaces are
given by Schur's test. 

\begin{prop}[Schur test for $p\geq 1$]
\label{Prop:SchurTestBigger1}
Let $1\leq p\leq \infty$, $\Lambda$
and $\Pi$ be countable sets,
$A=(a_{\lambda ,\rho })_{\lambda \in \Lambda, \rho \in \Pi}\in \bC ^{\Lambda \times 
\Pi}$ be a matrix satisfying
\begin{equation}
\label{eq:schur}
\sup _{\rho \in \Pi } \sum _{\lambda \in 
\Lambda }
|a_{\lambda , \rho } | \leq K_1
\quad \text{and}\quad
\sup _{\lambda \in \Lambda } \sum _{\rho \in 
\Pi }
|a_{\lambda , \rho} | \leq K_2\, . 
\end{equation}
Then $A$  is  a 
bounded operator from $\ell ^p (\Pi )$
to
$\ell ^p(\Lambda )$, and
$$
\norm{A}_{\cB(\ell^p(\Pi ),
\ell^p(\Lambda))} \leq K_1^{1/p'}K_2^{1/p}.
$$
\end{prop}

For a proof of Proposition
\ref{Prop:SchurTestBigger1},
see e.{\,}g. \cite[Lemma 6.1.2]{Gro01}.

\par

Because of \eqref{eq:schur} we let
\begin{equation}\label{Eq:SchurNorm}
\nm A{\Schur} =
\sup _{\rho \in \Pi} \sum _{\lambda \in 
\Lambda }
|a_{\lambda , \rho } | +
\sup _{\lambda \in \Lambda } \sum _{\rho \in 
\Pi }
|a_{\lambda , \rho } | ,
\end{equation}
for every $A\in \mathbb C^{\Lambda \times \Pi}$.

The following quasi-Banach space version of the Schur test just follows by using the triangle inequality. 

\begin{prop}[Schur test for $p\le 1$]
\label{Prop:SchurTestSmaller1}
Fix $0<p \leq 1$. Let
$A=(a_{\lambda ,\rho })_{\lambda \in \Lambda, \rho \in \Pi}\in
\bC ^{\Lambda \times  \Pi}$ be a
matrix satisfying
\begin{equation}
\label{eq:schur2}
\|A\|_{S-p} :=    \sup _{\rho \in \Pi }
\sum _{\lambda \in \Lambda }
|a_{\lambda , \rho } |^ p < \infty \, . 
\end{equation}
Then $A$ defines a bounded operator from $\ell ^p (\Pi )$ to $\ell
^p(\Lambda )$. The operator norm is bounded by 
$\|A\|_{S-p}^{1/p}$.
If $0<q \leq p\leq 1$, then
\begin{equation}
  \label{eq:h3}
\|A\|_{S-p}^{q/p} = \Big( \sup _{\rho \in \Pi} \sum _{\lambda \in \Lambda }
    |a_{\lambda , \rho } |^ p \Big) ^{q/p} \leq \sup _{\rho \in
      \Pi }  \sum _{\lambda \in \Lambda }
    |a_{\lambda , \rho } |^ q = \|A\|_{S-q} \, .
\end{equation}
\end{prop}

\noindent

Our treatment of the spectral invariance of pseudodifferential
operators  
relies on
spectral invariance properties of associated  convolution-dominated
matrices. Roughly speaking, a  matrix
$A \in \cB (\ell ^2(\Lambda ))$ is
called \emph{convolution-dominated},
if there exists a function $H $,
such that
\begin{equation}
\label{eq:1}
|a_{\lambda , \rho  }| \leq H (\lambda - 
\rho ) \qquad \forall
\lambda \in \Lambda, \rho \in \Pi  \, .
\end{equation}
The function $H$ is then called
an \emph{envelope} of $A$. By specifying a norm on envelopes, we
define a  particular class of convolution-dominated
matrices  as follows.

\par

\begin{definition}\label{Def:ConvDominatedMatrix}
Let $\Lambda , \Pi  \subseteq \rd $ be  
relatively separated and $0<p_0\leq 1$.  The set 
$
\mathcal{C}^{p_0}=\mathcal{C}^{p_0}(\Lambda, \Pi )
$
consists of all convolution-dominated matrices $A$ such that
\eqref{eq:1} holds for an envelope 
$H \in W(C_b, L^{p_0})$.
For  $A=(a_{\lambda, \rho})_{\lambda\in \Lambda , \rho \in \Pi}$ we set
\begin{align*}
\|A\|_{\mathcal{C}^{p_0}}
=\inf\{ \| H\|_{W(C_b,L^{p_0})}:
|a_{\lambda ,\rho}| \leq H (\lambda-\rho),
\forall
\lambda \in \Lambda , \rho \in \Pi \}.
\end{align*}
\end{definition}

\par


We set $\mathcal{C}^{p_0}(\Lambda)=\mathcal{C}^{p_0}(\Lambda, \Lambda)$. 
We observe that if $\Lambda=\Pi \subseteq \rd$ is a lattice, 
then the restriction of $H \in W(C_b, L^{p_0})$ to $\Lambda$ belongs to 
$\ell^{p_0}(\Lambda)$ due to Lemma
\ref{ConnWienAmalgSpaceSequenceSpace2}. In this case, $A\in \cC
^{p_0}(\Lambda )$, \fif\ there is a sequence $H\in \ell ^{p_0}(\Lambda
)$, such that $|A_{\lambda,\rho } | \leq H(\lambda  - \rho )$ for
all $\lambda, \rho \in \Lambda $. 

We note that $\| A\|_{\cC ^{p_0}}$ is a quasi-norm for $p_0<1$,  and
$\cC ^{p_0}$ is a quasi-Banach $*$-algebra (sometimes called a
$p$-algebra) with respect to addition and multiplication of
matrices. For  $p_0=1$,  $\cC ^1$ is a Banach $*$-algebra with norm
$\|\cdot \|_{\cC ^1}$. 

\par

For convolution-dominated operators
in Definition \ref{Def:ConvDominatedMatrix}
we state the following boundedness result.
Here we set $q'=\infty$ when $q\le 1$

\par

\begin{prop}\label{Prop:ConvDomOpCont}
Let $0 < p_0 \leq 1$, $\Lambda$ and $\Pi$ be
as in Definition \ref{Def:ConvDominatedMatrix}
and let
$A \in  \mathcal{C}^{p_0}(\Lambda ,\Pi )$.
Then $A$ is  bounded from $\ell^{q}(\Pi )$ to
$\ell^{q}(\Lambda)$ for every
$q\in [p_0,\infty]$, and 
\begin{align}\label{eq:Matrixnorm}
\|A\|_{\cB (\ell^q (\Pi ),\ell ^q(\Lambda ))} 
\leq
C\operatorname{rel}(\Lambda )^{\frac 1q}
\operatorname{rel}(\Pi )^{\frac 1{q'}}
\|A\|_{\mathcal{C}^{p_0}(\Lambda ,\Pi )},
\end{align}
where the constant $C>0$ only
depends on $d$.
\end{prop}

The result follows by
suitable combinations of H{\"o}lder's
and Young's inequalities. In
order to be self-contained we present a
proof in Appendix \ref{App:A}.


\par

\subsection{Invariance of the lower bound property 
on $\ell^p$ of convolution-dominated matrices}
Our main technical contribution is the so-called stability of
convolution-dominated matrices. By this is meant  the invariance of the
lower bound property of such matrices on $\ell^p$. 

\begin{tm}
\label{Thm:LowerBound}
Let $p_0 \leq  1$, $q\in [p_0,\infty ]$,
$\Lambda , \Pi \subseteq \rd$ be
relatively separated,
and $A \in  \mathcal{C}^{p_0}(\Lambda, \Pi)$. 
Assume that there exists
$p \in [p_0,\infty ]$ and $C_0>0$ such that 
\begin{equation}
  \label{eq:c4}
  \norm{Ac}_{p} \geq C_0 \norm{c}_{p}, \qquad
  c\in \ell ^{p} (\Pi ) \, .
\end{equation}
Then there exists a constant $C >0$
which is independent of $q$ such that
\begin{equation}
   \label{eq:c5}
  \norm{Ac}_{q} \geq C \norm{c}_q \qquad
  \text{ for all } c\in \ell
  ^q (\Pi) \, .
\end{equation}
\end{tm}

In other words, if $A$ in Theorem
\ref{Thm:LowerBound} is bounded from below on
some  $\ell ^p$ with $p\ge p_0$, then $A$
is bounded from below on $\ell ^q$ for all $q\in 
[p_0,\infty ]$.  Note that \eqref{eq:c4} is equivalent to saying that
$A$ is one-to-one on $\ell ^p(\Lambda )$ with closed range in $\ell
^p(\Pi )$. Thus if $A$ is  one-to-one  with closed range for
\emph{some} $p\in [p_0,\infty ]$, then it is one-to-one  with closed range for
\emph{all} $p\in [p_0,\infty ]$. 

Note that by \eqref{eq:Matrixnorm} $A$ in Theorem \ref{Thm:LowerBound}
is bounded  from $\ell ^p(\Pi )$
to $\ell ^p(\Lambda )$ for all $p\geq p_0$, 
hence the estimates \eqref{eq:c4}
and \eqref{eq:c5} really make sense.

The  proof of Theorem \ref{Thm:LowerBound} is 
modelled on  Sj\"ostrand's treatment of
Wiener's lemma for convolution-dominated
matrices.  It exploits the flexibility of 
Sj\"ostrand's methods to transfer lower
bounds for a matrix  from one value of
$p$ to all others.

\begin{rem}\label{rem:InvarianceLowerBound}
	The proof of the previous Proposition for $p_0=1$ can be found in  \cite[Proposition 8.1]{GOR15}. Hence we can restrict ourselves to the case $p_0<1$.	
  The cases to be considered are: 
\begin{itemize}
\item[(i)] from $p \geq 1$ to $q \geq 1$;

\item[(ii)] from $p \leq 1$ to $q<p$;

\item[(iii)] from $p\geq 1$ to $q < 1$;

\item[(iv)]  from $p< 1$ to $p< q$.
\end{itemize}
Case (iii) is a consequence of  cases (i) and (ii). If  $p\geq 1$ we can  assume $p=1$ on account of case (i). For $p=1$ the statement of case (iii) is included in case (ii). 
\end{rem}

We need some preparations for the proof. First let $\varphi \in
C^\infty (\mathbb R^d)$ with $0\le \varphi \le 1$ and $\supp (\varphi )
\subseteq B_2(0)$ be such that $\{ \varphi (\cdot -k)\}_{k\in \mathbb Z^d}$
is a partition of unity. Set $\varphi ^\varepsilon _k(x):=\varphi (\varepsilon x-k),\ 
x,y\in \mathbb R^d$. Then
\begin{equation}
\begin{alignedat}{1}
\sum _{k\in \mathbb Z^d}\varphi ^\varepsilon _k &= 1,
\quad
\abs{\varphi^\varepsilon_k(x) - \varphi^\varepsilon_k(y)}
\lesssim \varepsilon \abs{x-y},
\quad
0 \leq \varphi^\varepsilon_k \leq 1
\quad \text{and}
\\[1ex]
\Phi^\varepsilon &:= \sum_{k \in \zd}
\left( \varphi^\varepsilon_k \right)^2 \asymp 1 \, .
\end{alignedat}
\end{equation}

By combining these properties, we obtain
\begin{align}
\label{eq_improved_lip}
\abs{\varphi^\varepsilon_k(x) - \varphi^\varepsilon_k(y)}
\lesssim \min\{1, \varepsilon |x-y|\}, \qquad x,y\in \rd . 
\end{align}
For $k \in \zd$ and $\varepsilon >0$ let $\varphi^\varepsilon_k
c := {\varphi^\varepsilon_k}\big| _{\Pi} \cdot c, c=(c_{\rho})_{\rho \in \Pi}$, denote the multiplication operator $\varphi ^\varepsilon _k$. 
This multiplication operator enables us to get equivalent norms for
sequence spaces. 

\begin{lemma}\label{lemma:Normequivalence1}
	Let $\varepsilon >0$ and $\Lambda \subseteq \mathbb{R}^{2d}$ be
	relatively separated. Then for $0<q\leq \infty$ we get
	\begin{align}
\label{eq_sj_qq}
\Big(\sum_{k \in \zd} \norm{\varphi^\varepsilon_k a}_q^q\Big)^{1/q} 
\asymp \norm{a}_q,\qquad a\in \ell^q(\Pi ),
\end{align}
with the usual modifications in case $q=\infty$. 
\end{lemma}

For $1\leq q\leq \infty$, Lemma
\ref{lemma:Normequivalence1} was already 
proved in \cite{GOR15}, and the
other cases are obtained by similar
arguments. For completeness 
we present a proof for $q<1$
in Appendix \ref{App:A}.

\begin{lemma}\label{lemma:Normequivalence2}
	Let $0<p_0 \leq 1$,
	$p,q \in [p_0, \infty]$, $\varepsilon >0$
	and $\Lambda \subseteq \mathbb{R}^{2d}$ be relatively separated. Then
	\begin{align}
\label{eq_sj_b}
\sum_{k \in \zd} \norm{\varphi^\varepsilon_k a}_p^q 
\asymp \norm{a}_q^q ,
\qquad a \in \ell^q(\Pi ),
\end{align}
with the usual modifications in case $p=\infty$ or $q=\infty$. The
constants in \eqref{eq_sj_b} are independent of $p,q$,  but depend on $\varepsilon$.
\end{lemma}

In the  case $p,q \geq 1$ the claim  was already shown in \cite{GOR15}. 

\begin{proof}
Let $p,q \in [p_0, \infty]$ be arbitrary. 
	For fixed $\varepsilon >0$ we get
\begin{align*}
N := \sup _{k\in \zd}\# \supp({\varphi^\varepsilon_k}\big| _{\Pi}) 
= \sup_{k\in \zd} \#\set{\rho \in \Pi}{\varphi^\varepsilon_k(\rho) \not= 0} 
< \infty.
\end{align*}
Since  $q \geq p_0$ we have 
\begin{align*}
\norm{\varphi^\varepsilon_k a}_q
\leq \norm{\varphi^\varepsilon_k a}_{p_0}
\leq N^{1/p_0} \norm{\varphi^\varepsilon_k a}_\infty
\leq N^{1/p_0} \norm{\varphi^\varepsilon_k a}_p,
\qquad a \in \ell^\infty(\Pi ),
\end{align*}
and similarly
\begin{align*}
\norm{\varphi^\varepsilon_k a}_p
\leq N^{1/p_0}  \norm{\varphi^\varepsilon_k a}_q,
\qquad a \in \ell^\infty(\Pi ).
\end{align*}
As a consequence we obtain for $q \neq \infty$,
\begin{align}
\label{eq_sj_pq}
\Big( \sum_{k \in \zd} \norm{\varphi^\varepsilon_k a}_p^q \Big)^{1/q}
\asymp
\Big( \sum_{k \in \zd} \norm{\varphi^\varepsilon_k a}_q^q \Big)^{1/q},
\qquad a \in \ell^q(\Pi ).
\end{align}
with constants depending only on the minimal index 
$p_0$ and $\varepsilon$, but not on
$p$ and $q$. 

An application of Lemma 
\ref{lemma:Normequivalence1} on
\eqref{eq_sj_pq} yields the claim. The 
corresponding statement for $q=\infty$  follows 
similarly. 
\end{proof}

The technical  part of the proof consists of precise estimates for 
the Schur-type norms 
\begin{alignat}{2}
V^{\varepsilon,p}_{j,k} &:=
\norm{[A,\varphi^\varepsilon_k]\varphi
^\varepsilon_j}_{S-p}, &
\qquad 0<p\le 1,\  j,k &\in \zd 
\label{eq_def_V}
\intertext{and}
V^{\varepsilon }_{j,k} &:=
\norm{[A,\varphi^\varepsilon_k]\varphi
^\varepsilon_j}_{\Schur}, &
\qquad  j,k &\in \zd .
\label{eq_def_V2}
\end{alignat}
of the commutator $[A,\varphi^\varepsilon_k]
=
A \varphi ^\varepsilon _k - \varphi
^\varepsilon _k A$, when
$A \in \mathcal{C}^{p_0}(\Lambda, \Pi )$
and $\varphi ^\varepsilon _k$
is considered as a multiplication
operator.

\begin{lemma}\label{lemma:commutator}
	Let $\varepsilon >0$, $p_0\in (0,1]$,
	$p,q\in (p_0,\infty ]$
	be such that
	$ q \leq p $,
	$\Lambda ,\Pi \subset \rd$ be relatively
	separated and 
	suppose that $A=(a_{\lambda, \rho})_{\lambda \in \Lambda, \rho \in \Pi}
	\in C^{p_0}(\Lambda ,\Pi )$.
	Also let
	$V^{\varepsilon,p}_{j,k}$ and
	$V^{\varepsilon }_{j,k}$ be the Schur norms  given
	by \eqref{eq_def_V} and \eqref{eq_def_V2},
	$K := \max _x \Phi ^\varepsilon (x)^{-\min (1,p)}$. 
Assume 	that 
\begin{align*}
\norm{c}_p \leq \norm{Ac}_p, \qquad
c \in \ell^p(\Pi ).
\end{align*}
Then the following are true:
\begin{enumerate}
\item[(i)] If $p\le 1$, then
\begin{equation}
  \label{eq:h1}
  \norm{\varphi^\varepsilon_k c}_p^q \leq 
\norm{\varphi^\varepsilon_k Ac}_p ^q+
K^{q/p}  \sum_{j\in \zd} (V^{\varepsilon,p}_{j,k})^{q/p}
\norm{\varphi^\varepsilon_jc}_p^q, \qquad c\in
\ell^p(\Pi )\, .
\end{equation}

\vspace{0.1cm}

\item[(ii)] If $p> 1$, then
\begin{equation}
  \label{eq:h2}
  \norm{\varphi^\varepsilon_k c}_p \leq 
\norm{\varphi^\varepsilon_k Ac}_p+
K  \sum_{j\in \zd} V^{\varepsilon }_{j,k} \, 
\norm{\varphi^\varepsilon_jc}_p, \qquad c\in
\ell^p(\Pi )\, .
\end{equation}
\end{enumerate}
\end{lemma}

\begin{proof}
(i) We apply the triangle inequality for $p\leq 1$ (Lemma~\ref{plone}(ii))
and obtain 
\begin{align}
\nonumber
\norm{\varphi^\varepsilon_k c}_p^p & \leq \norm{A \varphi^\varepsilon_k c}_p^p
\leq \norm{\varphi^\varepsilon_k Ac}_p^p + \norm{[A,\varphi^\varepsilon_k] c}_p^p
\\
\nonumber
& \leq
\norm{\varphi^\varepsilon_k Ac}_p ^p+
 \sum_{j\in \zd} \norm{[A,\varphi^\varepsilon_k]\varphi^\varepsilon_j (\Phi^\varepsilon)^{-1} \varphi^\varepsilon_jc}_p^p
\\
\nonumber
& \leq
\norm{\varphi^\varepsilon_k Ac}_p ^p+
K \sum_{j\in \zd} V^{\varepsilon,p}_{j,k}
\norm{\varphi^\varepsilon_jc}_p^p\, .
\end{align}
Claim (i) now follows by raising this inequality  to the power
$q/p \leq 1$ and applying Lemma~\ref{plone}(i).

Assertion (ii) was proved  in \cite{GOR15}, Equ.~(36) with the same argument.
\end{proof}

Next we consider the matrix $V^\varepsilon $
with entries $(V^{\varepsilon,p} _{j,k})^{q/p}, j,k\in \zd $ in case
$p \leq 1$  and estimate its  $q/p$-Schur norm as $\varepsilon \to
0+$.  First we  prove the convergence of the entries of $V^\varepsilon
$.

\begin{lemma}\label{lemma:ConvergenceOfVepsilon}
Suppose that the hypothesis of Lemma \ref{lemma:commutator} hold true. Then for $\varepsilon \longrightarrow 0^+$
\begin{align}
\label{eq_V_unif}
\sup_{j,k \in \zd} V^{\varepsilon,p}_{j,k} 
\longrightarrow 0
\text{ if }
p\leq 1
\quad \text{ and } 
\quad 
\sup_{j,k \in \zd} V^{\varepsilon }_{j,k} 
\longrightarrow 0
\text{ if } 
p> 1.
%
\end{align}
\end{lemma}

\begin{proof}
The case $p\geq  1$ of \eqref{eq_V_unif} was proved in (38) of
\cite{GOR15}. The necessary adaptions for the proof of case $p \leq 1$
are as follows. 
We first note that the matrix entries of
$[A,\varphi^\varepsilon_k]
\varphi^\varepsilon_j$, for  $j,k \in \zd$,  are
\begin{align*}
([A,\varphi^\varepsilon_k]\varphi^\varepsilon_j)_{\lambda, \rho} = 
-a_{\lambda,\rho} \varphi^\varepsilon_j(\rho)(\varphi^\varepsilon_k(\lambda)-\varphi^\varepsilon_k(\rho)),
\qquad \rho \in \Pi , \lambda \in \Lambda.
\end{align*}
Using an envelope  $H \in W(C_b,L^{p_0})$ of $A$ and estimate 
\eqref{eq_improved_lip} for $\varphi _k^\varepsilon $, we bound the
entries of the commutator by 
\begin{align*}
&\abs{([A,\varphi^\varepsilon_k]\varphi^\varepsilon_j)_{\lambda,\rho}}^p
\lesssim
H(\lambda- \rho)^p
\min\{1, \varepsilon |\lambda - \rho |\}^p \, .
\end{align*}
Hence, if we define $H^{\varepsilon,p}(x) := H(x)^p
\min\{1,\varepsilon\abs{x}\}^p$, then by choice of $H$,
\begin{align*}
V^{\varepsilon,p}_{j,k} = 
\norm{[A,\varphi^\varepsilon_k]\varphi^\varepsilon_j}_{S-p}
\lesssim \rel(\Lambda ) ^{p /p_0 }  \norm{H^{\varepsilon,p}}_{W(C_b,L^{p_0/p })}.
\end{align*}
Since
{$H \in W(C_b,L^{p_0})$}
and $p_0 \leq p\leq 1$, it
follows with dominated convergence that 
$\norm{H^{\varepsilon,p}}_{W(C_b,L^{p_0/p} )} \longrightarrow 0$, as $\varepsilon \longrightarrow 0^+$.
This proves \eqref{eq_V_unif} for the case $p\leq 1$.
\end{proof}

Next we shall estimate
$V^{\varepsilon,p}_{j,k}$ in terms of
\begin{align}
\label{eq_def_menv}
\menv^{\varepsilon,q}(s) := \sum_{t \in \Zdst : 
\abs{\varepsilon
t-s}_\infty \leq 5} \sup_{z\in[0,1]^{d}+t} 
\abs{H (z) } ^q\, ,
\qquad s\in \mathbb Z^d. 
\end{align}
First we have the following.

\begin{lemma} \label{lemma:ConvergenceDelta}
Let $0<p_0 \leq q \leq 1$, $H \in W(C_b,L^{ p_0 })$ and
$\menv^{\varepsilon,q}(s)$ be given by
\eqref{eq_def_menv}.
Then 
\begin{align}
\label{bound_menv}
\sum_{s\in\Zdst,\abs{s}>6\sqrt{d}} \menv^{\varepsilon,q}(s) 
\longrightarrow 0,
\mbox{ as }\varepsilon \longrightarrow 0^+.
\end{align}
\end{lemma}

\begin{proof}
Since
$H \in W(L^\infty , L^{p_0})$,
we obtain  
\begin{align*}
\sum_{s\in\Zdst,\abs{s}>6\sqrt{d}} 
\menv^{\varepsilon,q}(s) 
\leq
\sum_{s\in\Zdst,\abs{s}_\infty>6} 
\menv^{\varepsilon,q}(s) 
\lesssim
\sum_{t \in \Zdst, 
\abs{t}_\infty>1/\varepsilon}
\sup_{z\in[0,1]^{d}+t} \abs{H (z) ^q }
\longrightarrow 0,
\end{align*}
as $\varepsilon \longrightarrow 0^+$.
\end{proof}

\begin{lemma}\label{lemma:estimateOfVepsilon}
Suppose that the hypothesis
of Lemma \ref{lemma:commutator} hold true with $q\leq 1$ and $\varepsilon \leq 1$,
and let $\menv^{\varepsilon,q}(s)$ be given by
\eqref{eq_def_menv}.
Then for $\abs{k-j}>4$ we have
\begin{align}
\label{eq_V_kj_new}
(V^{\varepsilon,p}_{j,k})^{q/p}&\leq 
\sup_{\rho \in \Pi} \sum_{\lambda \in \Lambda}
\abs{([A,\varphi^\varepsilon_k]\varphi^\varepsilon_j)_{\lambda,\rho}}^q
\lesssim \menv^{\varepsilon,q}(k-j) 
\text{ if } p\le 1,
\\ \label{eq_V_kj_new2}
(V^{\varepsilon }_{j,k})^{q}
&
\lesssim \menv^{\varepsilon,q}(k-j) \text{ if } p > 1.
\end{align}
\end{lemma}

\begin{proof}
Suppose that
$\abs{k-j}>4$.
Since $\varphi$ is supported
in $B_2(0)$, it follows that
$\varphi^\varepsilon_j(\rho)
\varphi^\varepsilon_k(\rho)=0$.

As a consequence,
the matrix
entries of $[A,\varphi^\varepsilon_k]
\varphi^\varepsilon_j$ simplify
into
\begin{align*}
([A,\varphi^\varepsilon_k]\varphi^\varepsilon_j)
_{\lambda,\rho}
= 
-a_{\lambda,\rho} 
\varphi^\varepsilon_j
(\rho)\varphi^\varepsilon_k(\lambda),
\qquad \rho \in \Pi , \lambda \in \Lambda \, .
\end{align*}
This gives
\begin{align*}
& \abs{([A,\varphi^\varepsilon_k]
\varphi^\varepsilon_j)_{\lambda,\rho}}^q
\leq 
\abs{H(\lambda- \rho)}^q
\varphi^\varepsilon_j(\rho)^q 
\varphi^\varepsilon_k(\lambda)^q . 
\end{align*}

Consequently, using \eqref{eq:h3} for
$\abs{k-j}>4$ and $q/p\leq 1$, we have 
\begin{align*}
(V^{\varepsilon,p }_{j,k} )^{q/p} &= \norm{[A,\varphi _k^\varepsilon  ] \varphi _j^\varepsilon}_{S-p}^{q/p}
\leq
\sup_{\rho \in \Pi}
\sum_{\lambda \in \Lambda}
\abs{([A,\varphi^\varepsilon_k]
\varphi^\varepsilon_j)_{\lambda,\rho}}^q
\leq
\sup_{\rho \in \Pi} \sum_{\lambda \in \Lambda}
\abs{H(\lambda- \rho)}^q
\varphi^\varepsilon_j( \rho)^q 
\varphi^\varepsilon_k(\lambda)^q \, .
\end{align*}
If $\varphi^\varepsilon_j(\rho) 
\varphi^\varepsilon_k(\lambda) \not = 0$, then
$\abs{\varepsilon \rho -  j} \leq 2$ and 
$\abs{\varepsilon \lambda - k} \leq 2$, whence 
\begin{align}
\label{eq_in}
\abs{\varepsilon (\lambda - \rho) + (j-k)}
\leq 4.
\end{align}
Hence 
\begin{align}
\label{eq_four_terms}
&\sup_{\rho \in \Pi} \sum_{\lambda \in \Lambda}
\abs{([A,\varphi^\varepsilon_k]\varphi^\varepsilon_j)
_{\lambda,\rho}} ^q
\lesssim
\sup_{\rho \in \Pi}
\sum_{
\stackrel{\lambda  \in  \Lambda}
{\abs{\varepsilon (\lambda - \rho)+(j-k)}\leq 4}
}
\abs{H(\rho-\lambda)}^q.
\end{align}

For fixed  $\varepsilon \leq 1$
we bound the  sum in \eqref{eq_four_terms} by 
\begin{align*}
&\sum_{\lambda \in \Lambda:\abs{\varepsilon (\lambda - \rho) + (j-k)}
                 \leq 4} \abs{H(\rho-\lambda)}^q \leq \sum_{t\in\Zdst}
\sum_{\stackrel{\lambda \in \Lambda : \abs{\varepsilon (\lambda - \rho) + (j-k)} \leq 4}{(\lambda - \rho) \in
[0,1]^{d}+t}} \abs{H(\rho-\lambda)}^q
\\
&\quad \lesssim \rel(\rho - \Lambda ) \sum_{t\in\Zdst :
  \abs{\varepsilon t + (j-k)}_\infty \leq 5} \sup_{z\in [0,1]^{d}+t} \abs{H (z)}^q
\\
&\quad\lesssim \menv^{\varepsilon,q}(k-j) \, ,
\end{align*}
since $\rel (\Lambda )$ is translation-invariant. Substituting  this bound in \eqref{eq_four_terms}, we  obtain \eqref{eq_V_kj_new}.
%
Analogously we obtain \eqref{eq_V_kj_new2}.
\end{proof}

\begin{lemma}
\label{lemma:ConvergenceOfShurNormOfVepsilon}
Suppose that the hypotheses
of Lemma  \ref{lemma:commutator} hold for $p\le 1$. Then
\begin{align}
\label{eq_schur_jk}
\sup_{k\in \zd} \sum_{j\in \zd} 
(V^{\varepsilon,p}_{j,k})^{q/p}+
\sup_{j\in \zd} \sum_{k\in \zd} 
(V^{\varepsilon,p}_{j,k})^{q/p} \longrightarrow 0,
\mbox{ as }\varepsilon \longrightarrow 0^+.
\end{align}
\end{lemma}

\begin{proof}
Let $\varepsilon \leq 1$ and   $\menv^{\varepsilon,q}$ be defined as
in~\eqref{eq_def_menv}. 
Fix $j \in \zd$ and use  Lemma \ref{lemma:estimateOfVepsilon}
to estimate
\begin{align*}
&\sum_{k: \abs{k-j} > 6\sqrt{d}} (V^{\varepsilon,p}_{j,k})^{q/p} \lesssim
\sum_{k: \abs{k-j} > 6\sqrt{d}} 
\menv^{\varepsilon,q}(k-j).
\end{align*}
For the sum over $\{k\in \zd : \abs{j-k} \leq 6 \sqrt{d}\}$ we use the 
bound 
\begin{align*}
&\sum_{k: \abs{j-k} \leq 6\sqrt{d}} (V^{\varepsilon,p}_{j,k})^{q/p}
\leq
\#\{k: \abs{j-k} \leq 6\sqrt{d}\} \sup_{s,t} (V^{\varepsilon,p}_{s,t})^{q/p}
\lesssim \sup_{s,t} (V^{\varepsilon,p}_{s,t})^{q/p}.
\end{align*}
Hence,
\begin{align*}
\sum_{k \in \zd } (V^{\varepsilon,p}_{j,k})^{q/p}
\lesssim \sup_{s,t} (V^{\varepsilon,p}_{s,t})^{q/p} + \sum_{\abs{s}>6\sqrt{d}} \menv^{\varepsilon,q}(s),
\end{align*}
which tends to 0 uniformly in  $j$ as $\varepsilon \rightarrow 0^+$
by Lemmas \ref{lemma:ConvergenceOfVepsilon} and  \ref{lemma:ConvergenceDelta}.  
 The convergence of the first term in~\eqref{eq_schur_jk} follows in exactly the same way by interchanging the roles of $j$ and $k$. 
%
\end{proof}

With the auxiliary results at hand, we now  prove 
Theorem \ref{Thm:LowerBound}.

\begin{proof}[Proof of Theorem \ref{Thm:LowerBound}]

As already mentioned we restrict ourselves to the case $p_0<1$, cf.\ Remark \ref{rem:InvarianceLowerBound}.
Since $W(C^0; L^{p_0})(\rd) \subseteq W(L^{\infty}; L^1)(\rd)$ an
application of Proposition 8.1 of \cite{GOR15} provides the claim in
the Banach space case $p,q \geq 1$. 

\textbf{Case $p \leq 1$ and $q < p$} so that $q/p<1$.
After multiplying $A$ with a constant, we may  assume that
\begin{align*}
\norm{c}_p \leq \norm{Ac}_p, \qquad c \in \ell^p(\Pi),
\end{align*}
since $A$ is bounded from below on $\ell^p(\Pi )$ by assumption~\eqref{eq:c4}.
By  Lemma \ref{lemma:commutator}, \eqref{eq:h1},  we obtain
\begin{equation}
  \label{eq:h5}
\sum _{k\in \bZ^{d} }  \norm{\varphi^\varepsilon_k c}_p^q 
  \leq   \sum _{k\in \bZ^{d} } \norm{\varphi^\varepsilon_k Ac}_p ^q +
K ^{q/p}\sum _{k\in \bZ^{d} } \sum_{j\in \zd} (V^{\varepsilon,p}_{j,k})^{q/p}
\norm{\varphi^\varepsilon_jc}_p^q, \quad c\in \ell^p(\Pi)  \, ,
\end{equation}
for some $K>0$.
 According to Lemma \ref{lemma:ConvergenceOfShurNormOfVepsilon} we may choose $\varepsilon >0$
 such that
 $$
  K ^{q/p}  \sup _{j\in \bZ^{d}  } \sum_{k\in \bZ^{d}} (V^{\varepsilon,p}_{j,k})^{q/p} <
\frac{1}{2} \, .
$$
 Using this bound in \eqref{eq:h5} and Proposition~\ref{Prop:SchurTestSmaller1},
we obtain that 
\begin{align*}
\sum_{k \in \zd} \norm{\varphi^\varepsilon_k c}_p^q
\leq
\sum_{k \in \zd}\norm{\varphi^\varepsilon_kAc}_p^q
+ 1/2 \sum_{k \in \zd} \norm{\varphi^\varepsilon_k
c}_p^q \, .
\end{align*}
 Hence,
\begin{align}
\label{eq_sj_a}
\left(\sum_{k \in \zd} \norm{\varphi^\varepsilon_k c}_p^q\right)^{1/q}
\leq
2^{1/q} \left(\sum_{k \in \zd}\norm{\varphi^\varepsilon_k Ac}_p^q\right)^{1/q}.
\end{align}

Using the 
equivalent norm  of Lemma \ref{lemma:Normequivalence2} in
\eqref{eq_sj_a}, we   deduce that for all $p_0 \leq q < p$
\begin{align*}
\norm{c}_q \lesssim \norm{Ac}_q,
\end{align*}
with a constant independent of $q$ (since $2^{1/q} \leq
2^{1/p_0}$). This completes the proof of the case $p_0\leq q\leq p$. 

\textbf{Case $p<1$ and $p<q$.}
Again
we may assume that 
\begin{align*}
\norm{c}_p \leq \norm{Ac}_p, 
\quad
c \in \ell^p(\Pi).
\end{align*}
According to Lemma \ref{lemma:ConvergenceOfShurNormOfVepsilon}
in case $q=p$  
we may choose $\varepsilon >0$ 
 such that
 \begin{align} \label{eq:estimate1}
   K  \sup _{j\in \zd  } \sum_{k\in \zd} V^{\varepsilon,p}_{j,k} <
\frac{1}{2}
\quad \text{ and }
\quad 
  K  \sup _{k\in \zd  } \sum_{j\in \zd} V^{\varepsilon,p}_{j,k} <
\frac{1}{2}
\, .
\end{align}
Due to Lemma \ref{lemma:commutator} we have, for $c \in \ell
^p(\Pi)$, 
\begin{align*}
 	\norm{\varphi^\varepsilon_k c}_p^p 
 	\leq \norm{\varphi^\varepsilon_k Ac}_p ^p+
K \sum_{j\in \zd} V^{\varepsilon,p}_{j,k} \norm{\varphi^\varepsilon_jc}_p^p.
\end{align*}
We set $a^{\varepsilon}_k= \norm{\varphi_k^{\varepsilon}c}_p^p$,
$b^{\varepsilon}_k=\norm{\varphi_k^{\varepsilon}Ac}_p^p$ and let
$V^{\varepsilon}$ be the operator associated to  the matrix $V^{\varepsilon,p}_{j,k}$. Then the previous inequality can be written as
\begin{align*}
	a^{\varepsilon}_k \leq b^{\varepsilon}_k + K (V^{\varepsilon}a^{\varepsilon})_k.
\end{align*}
 We take the $q/p-$norm of the previous estimate, apply the triangle inequality first. Next we apply
Proposition \ref{Prop:SchurTestBigger1}, use
\eqref{eq:estimate1} and get
  \begin{align*}
  		\norm{a^{\varepsilon}}_{q/p} \leq \norm{b^{\varepsilon}}_{q/p}+ K \norm{V^{\varepsilon}a^{\varepsilon}}_{q/p}
  		\leq  \norm{b^{\varepsilon}}_{q/p}+ \frac{1}{2} \norm{a^{\varepsilon}}_{q/p}.
 \end{align*}
 In other words we have $\norm{a^{\varepsilon}}_{q/p} \leq 2
 \norm{b^{\varepsilon}}_{q/p}$. Undoing the abbreviations, this means
 that
 \begin{align*}
 \left(\sum_{k \in \zd} \norm{\varphi^{\varepsilon}_k c}_p^q
  \right)^{1/q} = \norm{a^\varepsilon}_p^{1/p}\leq
2^{1/p}   \norm{b^\varepsilon}_p^{1/p} =  2^{1/p} \left(\sum_{k \in \zd} \norm{\varphi^{\varepsilon}_k A c}_p^q \right)^{1/q}.
\end{align*}
An application of the norm equivalence of Lemma
\ref{lemma:Normequivalence2} provides the claim 
\begin{align*}
	\norm{c}_q \lesssim \norm{Ac}_q.
\end{align*}
\end{proof}

As observed in~\cite[Rem.~A.2]{GOR15} the  lower bound guaranteed by Theorem
\ref{Thm:LowerBound} is   uniform for all $p$.

As an immediate consequence of Theorem 
\ref{Thm:LowerBound} we get a first spectral 
invariance result of a matrix $A \in 
\mathcal{C}^{p_0}(\Lambda, \Pi)$ by means of standard functional analytical arguments.

\begin{cor}\label{Prop:SpectralInvarianceMatrix}
Let $p_0 \in (0,1]$,
$\Lambda , \Pi \subseteq \rd$ be relatively separated,
and
suppose that $A\in C^{p_0}(\Lambda ,\Pi )$
is invertible from $\ell^p(\Pi)$ to
$\ell^p(\Lambda)$ for some 
$p\in [p_0,\infty]$. Then the following is true:
\begin{enumerate}
\item[(i)] $A$ is injective from $\ell^q(\Pi)$ to
$\ell^q(\Lambda)$ for all $q \in 
[p_0,\infty ]$. 

\item[(ii)] $A$ is
invertible from $\ell^q(\Pi)$ to
$\ell^q(\Lambda)$ for all $q \in [p,\infty )$.
\end{enumerate}
\end{cor}
\begin{proof}
  As already observed, $A: \ell ^p(\Pi ) \to \ell ^p(\Lambda )$ satisfies
  the stability condition $\|c\|_p \lesssim \|Ac\|_p $, \fif\ $A$ is
  one-to-one on $\ell ^p(\Pi )$ and has closed range in $\ell
  ^p(\Lambda )$. Thus, if $A$ is invertible on some $\ell ^p(\Pi)$,
  then by Theorem~\ref{Thm:LowerBound} $A$ satisfies the stability
  condition~\eqref{eq:c4} for all $q\geq p_0$. Consequently $A$ is
  one-to-one on all $\ell ^q(\Pi )$ for $q\geq p_0$, which is (i).

  If $q\geq p$, then $\ell ^p(\Lambda )$ is dense in  $\ell ^q(\Lambda
  )$. Since $A$ is onto $\ell ^p(\Lambda )$, we obtain  $\ell ^p(\Lambda ) =
  A\ell ^p(\Pi) \subseteq \ell ^q(\Lambda )$, and $A$  has also dense range
  in $\ell ^q(\Lambda )$. Consequently $A$ is onto $\ell ^q(\Lambda )$
  and thus  invertible on  $\ell ^q(\Lambda )$.   
\end{proof}

In the case, when  $A$  in invertible is the Hilbert space
$\ell ^2$, the above  results are already contained in~\cite[Theorem 4.6 and
Theorem 8.5]{MoSu17} and in \cite[Theorem 3.9]{CoGi22} and in
~\cite{DaSo22_2}.

Note that Corollary~\ref{Prop:SpectralInvarianceMatrix} asserts only
  that the invertibility of $A$ on $\ell ^p$ implies the invertibility
  on the larger space $\ell ^q$, $q>p$. To obtain the same conclusion
  for the smaller spaces $\ell ^q, q<p$, we need to refine our
  arguments. We will show that 
the inverse of $A$ has an envelope belonging to the same Wiener amalgam space than the envelope of  $A$. The main idea of the proof of the following Theorem is taken from \cite{Sjo95}. 

\begin{tm}\label{tm:SpectralInvariance}
  Let $0< p_0 \leq 1$, $\Lambda, \Pi \subseteq \rd $ be relatively
  separated. 
  Suppose $A \in \mathcal{C}^{p_0}(\Lambda, \Pi )$ is invertible 
  from $\ell ^p(\Pi)$ to  $\ell ^p(\Lambda )$ for some $p\in [p_0,
  \infty ]$.
  Then  $A^{-1}\in  \mathcal{C}^{p_0}(\Pi, \Lambda )$. 
\end{tm}

\par

\par

\begin{proof}
By assumption  the matrix 
$
A=(a_{\lambda, \rho})_{\lambda \in \Lambda, \rho \in \Pi}
$ 
has an envelope $H_1 \in W(C_b, L^{p_0})$ so that $|a_{\lambda ,
  \rho }| \leq H_1(\lambda -\rho )$.  
We need to  prove that the matrix $A^{-1}$ also has an envelope $H_2 \in W(C_b,L^{p_0})$.
 For this we recall the notation used in the previous lemmas:  $K= \max _x\Phi ^\varepsilon
 (x)^{-\min (1,p)}$ from Lemma~\ref{lemma:commutator}, and
$$
\widetilde V^\varepsilon
=
\begin{cases}
\big ( (KV^{\varepsilon,p }_{j,k})^{p_0/p}
\big )_{j,k\in \mathbb Z^d}, & p\le 1,
\\[1ex]
\big ( (KV_{j,k}^{\varepsilon})^{p_0}
\big )_{j,k\in \mathbb Z^d}, & p> 1,
\end{cases}
$$
where $V^{\varepsilon,p}_{j,k}$ and
$V^{\varepsilon }_{j,k}$ are defined as in
\eqref{eq_def_V} and \eqref{eq_def_V2}.

Given  $c\in \ell ^p(\Pi )$, the
 sequences $a $ and $a_{A}$ are defined by  
$$
a_k = \nm {\varphi ^\varepsilon _kc}p^{p_0}
\quad \text{and}\quad
a_{A,k} = \nm {\varphi ^\varepsilon _kAc}p^{p_0},
\qquad
k\in \mathbb Z^d.
$$
With this notation  Lemma \ref{lemma:commutator} (and 
    additionally Lemma \ref{plone} if $p>1$) says that 
\begin{equation}
  \label{eq:co4}
  a \leq a_A+\widetilde{V}^\varepsilon a \, .
\end{equation}

\par

Furthermore, Lemma \ref{lemma:estimateOfVepsilon} can be reformulated
as saying that  $\widetilde V^\varepsilon$ is convolution-dominated. 
In fact, 
recall from \eqref{eq_def_menv} that 
\begin{align*}
	\Delta ^{\varepsilon,p_0}(s):= \sum_{t \in \zd:|\varepsilon
  t-s|_{\infty} \leq 5} \sup_{z\in [0,1]^d+t} |H_1(z)|^{p_0} \, ,
\end{align*}
and  define $\Psi ^\varepsilon$ 
by $\Psi^{\varepsilon}(s):= \Delta ^{\varepsilon,p_0}(s)$ for all $s \in \zd$ with $|s| > 6 \sqrt{d}$ 
 and for all $s \in \zd$ with $|s| \leq 6 \sqrt{d}$
 by  $\Psi^{\varepsilon}(s):= 	\left( \sup_{j,k \in \zd} |V_{j,k}^{\varepsilon,p}| \right)^{p_0/p}$ 
 if $p \leq 1$ and
 by  $\Psi^{\varepsilon}(s):= 	\left( \sup_{j,k \in \zd}
   |V_{j,k}^{\varepsilon}| \right)^{p_0}$ if $p \geq 1$. 
In view of the normalization $K$, 
$\tilde{V}^\varepsilon$   possesses the 
envelope $K^{p_0/ \min(1,p)}\Psi^{\varepsilon}$. 
This means that  
\begin{alignat}{2}
    (V^{\varepsilon,p}_{j,k})^{p_0/p}
    &\lesssim \Psi ^\varepsilon (k-j ), & \qquad
    p&\le 1,\ j,k\in \zd \, ,
\label{eq:co3}
\intertext{and}
    V^{\varepsilon }_{j,k}
    &\lesssim \Psi ^\varepsilon (k-j ), & \qquad
    p&> 1,\ j,k\in \zd \, .
\label{eq:co3A}
\end{alignat}
Our next goal is to  represent 
$(I-\widetilde{V}^{\varepsilon})^{-1}$
as a  Neumann series.
For this we choose $\varepsilon >0$ such that 
\begin{equation}
  \label{eq:co12}
  \|\widetilde{V}^\varepsilon \|_{S-1} \leq K^{p_0/\min(1,p)} \|\Psi ^\varepsilon \|_1 \leq
\frac{1}{2}. \,
\end{equation}

This is possible due to Lemma \ref{lemma:ConvergenceDelta} and Lemma
\ref{lemma:ConvergenceOfVepsilon}, since $\sharp \{s\in \zd: |s|\leq
6\sqrt{d}\}$ is finite and depends only on the dimension $d$. 

As a consequence the geometric series  $\widetilde{W}:=  \sum
_{k=1}^\infty \left( \widetilde{V}^{\varepsilon}\right)^k$ converges
in the $\|.\|_{S-1}$-norm and  we obtain 
$$
(\mathrm{I} - \widetilde{V}^\varepsilon )\inv = \sum _{k=0}^\infty \left( \widetilde{V}^{\varepsilon}\right)^k =
\mathrm{I} + \widetilde{W} \, . 
$$
  Since all entries of $\widetilde{V}^\varepsilon$ are non-negative by definition,
$\widetilde{W}$ also has only non-negative entries and preserves
(pointwise) inequalities. Moreover, since $\widetilde{V}^\varepsilon$ is
convolution-dominated, so is $\widetilde{W}$, and by \eqref{eq:co12} there
exists an envelope  $W\in \ell ^1(\zd )$ such that
$$
\widetilde{W}_{jk} \leq W(k-j) \qquad j,k \in \zd \, .
$$
Now \eqref{eq:co4} yields
\begin{equation}
  \label{eq:co6}
  a = (\mathrm{I} +\widetilde{W})(\mathrm{I} - \widetilde{V}^\varepsilon ) a \leq
  (\mathrm{I} + \widetilde{W}  )a_A \, ,
\end{equation}
or entrywise
\begin{equation}\label{eq:co7}
\|\varphi _k^\varepsilon c\|_p^{p_0}
\leq
\|\varphi _k^\varepsilon Ac\|_p^{p_0} +
\sum _{j\in \zd } W(k-j) \|\varphi _j^\varepsilon Ac\|_p^{p_0} \, .  
\end{equation}
Since $A$ is assumed to be invertible as a map from
$\ell ^p(\Pi )$ to $\ell ^p(\Lambda )$, there
exist $b _\lambda \in \ell ^p(\Pi  )$ such that $Ab_\lambda =
\delta_\lambda $, whence the matrix $B$ with entries $b_{\rho \lambda
} = (b_\lambda ) _\rho $ is the inverse of $A$. Using $b_\lambda $ in
\eqref{eq:co7} we obtain
\begin{equation}
  \label{eq:co8}
\|\varphi _k^\varepsilon b_\lambda \|_p^{p_0} \leq \|\varphi _k^\varepsilon \delta
_\lambda \|_p^{p_0} +
\sum _{j\in \zd } W(k-j) \|\varphi _j^\varepsilon \delta _\lambda \|_p^{p_0} \, .  
\end{equation}
Note that
$$
\|\varphi _k^\varepsilon \delta_\lambda \|_p^{p_0} = \big( \sum _{\rho \in
  \Lambda } \varphi (\varepsilon \rho -k)^p \delta _\lambda (\rho )
\big)^{p_0/p} = \varphi (\varepsilon \lambda -k )^{p_0} \, .
$$
Let $\lambda  \in \Lambda $, $\rho \in \Pi$. For the off-diagonal decay it
suffices to consider only indices satisfying $\varepsilon |\lambda -\rho |
>4$. Choose $k_\rho \in \zd $ such that
$$
|\varepsilon \rho -k_\rho |< 2 \quad \text{  and  } \quad \varphi (\varepsilon
\rho -k_\rho ) ^{p_0} \geq c
$$
for some constant $c$  (in fact, by \eqref{eq_sj_partq} in the proof of Lemma \ref{lemma:Normequivalence1} $c=\eta \inv $).
Then
$$
|\varepsilon \lambda -k_\rho | \geq \varepsilon |\lambda -\rho | - |\varepsilon
\rho -k_\rho | > 2 \, ,
$$
and consequently $\varphi (\varepsilon \lambda -k_\rho ) = 0$ in
\eqref{eq:co8} since $\text{supp}(\varphi) \subseteq B_2(0)$. Now  \eqref{eq:co8} simplifies
to  
\begin{equation}
  \label{eq:co9}
  c |(b_ \lambda )_\rho |^{p_0} \leq \|\varphi _{k_\rho } ^\varepsilon b_\lambda
  \|_p^{p_0} \leq \sum _{j\in \zd } W(k_{\rho} -j) \varphi (\varepsilon \lambda
  -j)^{p_0} \leq \sum _{j: |j-\varepsilon \lambda |<2} W(k_\rho - j) \, .
\end{equation}
This inequality suggests the following envelope for $B=A\inv $. Let
\begin{equation}
  \label{eq:co10}
H (x) = \sum _{l\in \zd : |l- \varepsilon x| <4} W(l). 
\end{equation}
To obtain a continuous envelope, we use a cut-off function  $\psi \in
C^{\infty}_c(\rd)$ satisfying $0 \leq \psi \leq 1$ and $\psi(x)=1$ for
$|x|\leq 4$ and set
\begin{align*}
    \tilde H(x)= \sum_{l \in \mathbb{Z}^d} W(l)\psi(\varepsilon x-l) \, . 
\end{align*}
Since $\psi $ has compact support and $W(C_b,\ell ^1)$ is translation
invariant, 
we have 
\begin{align*}
    \|\tilde H\|_{W(C_b,\ell^1)}
    \leq 
    \sum_{l\in \mathbb{Z}^d} |W(l)| \sup _{l\in \zd } \|\psi(\varepsilon \cdot -l)\|_{W(C_b, \ell^1)}
    \lesssim 
    \|W\|_{\ell^1} \, ,
\end{align*}
therefore $\tilde H \in W(C_b,\ell^1)$. 

Furthermore, since $|k_\rho -j-\varepsilon(\lambda-\rho)| \leq |k_\rho -\varepsilon \rho | + |\varepsilon \lambda -j| <  4$, \eqref{eq:co10} says that
\begin{equation}
  \label{eq:co11}
  |b_{\rho \lambda } |^{p_0} \leq H (\rho - \lambda)  
 \leq \tilde{H}  (\rho -\lambda )  
  \, ,
\end{equation}
or
$$
|b_{\rho \lambda } |= |(b_ \lambda )_\rho  | \leq \tilde{H} (\rho - \lambda
)^{1/p_0} \, ,
$$
and $\tilde{H} ^{1/p_0} \in W( C_b, L^{p_0})$,  as claimed.
\end{proof}

This theorem enables us to extend  the spectral invariance
result of Corollary~\ref{Prop:SpectralInvarianceMatrix} to the case
$p_0 \leq q\leq p$ as follows. 

\begin{tm}\label{Thm:SpectralInvarianceMatrix_new}
Let $p_0 \in (0,1]$, $\Lambda, \Pi \subseteq \rd$ be relatively separated and
$A \in C^{p_0}(\Lambda, \Pi)$ be
invertible from $\ell^p(\Pi )$ to $\ell^p(\Lambda)$ for some $p\in [p_0,\infty]$. Then $A$ is invertible from  $\ell^q(\Pi )$ to $\ell^q(\Lambda)$ for all $q \in [p_0,\infty)$.
\end{tm}

\begin{proof}
According to    Corollary
\ref{Prop:SpectralInvarianceMatrix}(ii) $A$ is invertible on $\ell ^q$
for $q\geq p$ and $A$ is one-to-one on all
$\ell ^q$.  So 
it remains to prove that $A$ is surjective for
$p_0 \leq q < p$. We assume that $A$ is invertible
on $\ell^p$ with inverse $A^{-1}$. 
Hence for $u \in \ell^q(\Lambda) \subseteq
\ell ^p(\Lambda)$ there is some $c \in \ell^p(\Pi )$
with $c=A^{-1}u$. 
By Theorem \ref{tm:SpectralInvariance}  $A^{-1}$ is bounded  on
$\ell^q$, consequently,  since $u\in \ell ^q$ we have $c=A\inv u \in
\ell ^q$. Thus $A$ is    
onto $\ell^q(\Lambda)$.
\end{proof}

\par

The following consequence explains why the statement of
Theorem~\ref{Thm:SpectralInvarianceMatrix_new} is referred to as the
spectral invariance property.

\begin{cor}\label{tm:Spectralinvariance}
Let $0<p_0 \leq 1$ be arbitrary. If $A \in
\mathcal{C}^{p_0}(\Lambda )$ and $A$ is invertible on
$\ell^{p}(\Lambda )$ for some $p\in [p_0,\infty]$, then $A^{-1} \in
\mathcal{C}^{p_0}(\Lambda )$ and
 \begin{align*}
 	Sp_{\mathcal{B}(\ell^{q})}(A) = Sp_{ \mathcal{C}^{p_0}}(A) \qquad \text{ for all } q\in [p_0,\infty),
 \end{align*}
  where $Sp_{\mathcal{A}}(A)$ denotes the spectrum of $A$ in the algebra $\mathcal{A}$.
\end{cor}

In the literature many variations of this spectral invariance result exist. For an  overview of those variations we refer to \cite{Gro10}. Here we just want to mention the following ones: 
 In case $p_0=1$ and $p=q=2$ the previous theorem also holds in the weighted case. It was proved by Baskakov e.g. in \cite{Ba90, Ba97} and by Sj\"ostrand \cite{Sjo95} in the unweighted case. 

\section{Spectral Invariance of Pseudodifferential Operators}

The aim of this section is to transfer the results on the  spectral
invariance of matrices to 
pseudodifferential operators on unweighted modulation spaces.
In the proofs we mainly follow \cite{Gro06} and
\cite{GR08} and replace an arbitrary spectrally invariant Banach
algebra of matrices by the quasi-Banach algebra $\cC ^{p_0}$.

First we
list up all needed definitions and properties to reach that aim. We
start with recalling the modulation spaces.

\subsection{Modulation Spaces $M^{p,q}$.}

 For the definition of those spaces we need the short-time Fourier
 transform, which  we recall now.

Let $g \in \s \setminus \{0\}$ be fixed. For every
$f \in \mathcal{S}'(\mathbb{R}^d)$, the
\textit{short-time Fourier transform} $V_gf$ is
the function on $\mathbb{R}^{2d}$ defined by the formula
\begin{align}\label{def:STFT}
	V_gf(x,\xi)
	:=
	\langle f,g(\cdot -x)e^{2\pi i   \xi \cdot \, \cdot } \rangle.
\end{align}
Here $\scal \cdo \cdo$ is the unique extension of the
$L^2$ scalar product on $\mathcal S (\mathbb R^d)
\times \mathcal S (\mathbb R^d)$
into $\mathcal S ' (\mathbb R^d)
\times \mathcal S (\mathbb R^d)$. We observe that
if $f\in L^p(\mathbb R^d)$ for some $p\in [1,\infty ]$,
then $V_gf$ is given by \eqref{def:STFTIntro}.
If $g$ and $f$ are both defined on
$\mathbb{R}^{2d}$, then $V_gf$ is a function on
$\mathbb{R}^{4d}$.

\par

We recall that for 
$0<p,q \leq \infty$ and fixed $g \in \s \setminus \{0\}$,
the \emph{modulation space} $M^{p,q}(\rd)$ consists of all $f \in \mathcal{S}'(\rd)$
such that \eqref{Eq:ModNorm} is finite. If $1 \leq p,q \leq \infty$,
then $M^{p,q}(\rd)$ is a Banach space with norm
$\| \cdo \|_{M^{p,q}}$.
Otherwise $M^{p,q}(\rd)$ is a quasi-Banach space with
quasi-norm $\| \cdo \|_{M^{p,q}}$. We write $M^p=M^{p,p}$.


It is well-known  that the definition of $M^{p,q}(\rd)$ is independent
of the choice of the window function $g \in \s$~\cite{Gro01}. For $p<1$ or 
$q<1$ the proof of this fact can be found in \cite{GaS04}.  
Additionally,
the Schwartz space  $\s$ is dense in $M^{p,q}(\rd)$ in the case
$p,q <\infty$, cf. \cite[Remark 14]{GaS04}. 

Due to  \cite[Theorem 3.6]{Kob06} and \cite[Theorem 3.4.]{GaS04} the
following continuous embeddings of modulation spaces hold.  

\begin{prop}
\label{Prop:EmbeddingsModulationSpaces}
Let $0<p_1 \leq p_2 \leq \infty$
and $0<q_1 \leq q_2 \leq \infty$. Then
$$
\mathcal{S}(\rd) \hookrightarrow 
M^{p_1,q_1}(\rd) \hookrightarrow 
M^{p_2,q_2}(\rd)
\hookrightarrow \mathcal{S}'(\rd).
$$
\end{prop}

\subsection{Gabor frames}

For the definition of Gabor frames
it is convenient to use
time-frequency shifts $\pi(z) f$
of $f \in \mathcal{S}'(\rd)$
given by
\begin{align*}
(\pi(z) f)(t)
:=
e^{2\pi i \xi \cdot t} f(t-x), 
\qquad z=(x,\xi) \in \mathbb R^{2d},\ t\in \rd.
\end{align*}
The \emph{Gabor system} with
respect to the \emph{(Gabor) atom}
$g \in M^1(\rd )\backslash \{0\}$
and lattice
$\Lambda \subseteq \mathbb R^{2d}$
is given by
\begin{align*}
\mathcal{G}(g,\Lambda)
=
\{ \pi(\lambda)g: \lambda \in \Lambda \}.
\end{align*}
Then the \emph{analysis operator} and
\emph{synthesis operator}
$$
C_g\, :\, M^\infty (\rd )
\to
\ell ^\infty (\Lambda)
\quad \text{and}\quad
D_g\, :\, \ell ^\infty (\Lambda)
\to
M^\infty (\rd ),
$$
respectively, with respect to $g$ and
$\Lambda$, are given by
$$
C_gf = \big (\scal f{\pi (\lambda )g} \big )
_{\lambda \in \Lambda}
\quad \text{and}\quad
D_gc = \sum _{\lambda \in \Lambda}
c_\lambda \pi (\lambda )g ,
$$
when $f\in M^\infty (\rd )$ and
$c=(c_\lambda )_{\lambda \in \Lambda}\in
\ell ^\infty(\Lambda)$. Here the series converges
in $\mathcal S'(\rd )$.

 The (Gabor) frame operator 
$S=S_{g_1,g_2,\Lambda}: M^{\infty}(\rd ) \to 
M^{\infty}(\rd)$, $g_1,g_2 \in M^1(\rd)$ is defined  by
\begin{align}
S_{g_1,g_2,\Lambda}f:=\sum_{ \lambda \in \Lambda}
\langle f,\pi(\lambda ) g_1 
\rangle \pi(\lambda ) g_2.
\label{Eq:FrameOp}
\end{align}
If $g_1=g_2=g$ we write $S_{g,\Lambda}$ instead of $S_{g_1,g_2,\Lambda}$. It follows that $C_g, D_g$ and $S_{g,\Lambda}$
are well-defined and continuous
(see e.{\,}g. Chapters 11-14
in \cite{Gro01}).
Let $g \in \s \backslash\{0\}$.
If $S_{g,\Lambda}$	is bounded and invertible on 
$L^2(\rd)$, then we call $\mathcal{G}(g,\Lambda)$ 
\emph{Gabor frame}. For rectangular lattices 
$\Lambda =\alpha \zd \times \beta \zd$ the 
existence of Gabor frames is well understood, see 
\cite{Da90, Gro01, Ja95, Wa93}. 

For every  Gabor frame  $\mathcal{G}(g,\Lambda)$ over a lattice,
there exists a  
dual window $\gamma = S^{-1}g \in L^2(\rd)$, so that every $f$ can be
expanded into a  Gabor expansion
\begin{align}\label{GaborExpansionOfFunction}
f&=D_g C_{\gamma}f=\sum_{\lambda \in \Lambda} \<f, \pi(\lambda) \gamma\> \pi(\lambda) g,
\\[1ex] \label{GaborExpansionOfFunction2}
f&=D_{\gamma} C_g f=\sum_{\lambda \in \Lambda} \<f, \pi(\lambda) g\> \pi(\lambda) \gamma.
\end{align}
For $g\in \mathcal S(\rd )$ a fundamental result of
Janssen~\cite{Ja95} asserts  that also 
$\gamma \in \mathcal S (\rd)$.
Then the expansion formulas \eqref{GaborExpansionOfFunction} and
\eqref{GaborExpansionOfFunction2} hold  for every $f\in \cS '(\rd )$  
with weak-$*$-convergence. 

A Gabor frame $\mathcal{G}(g,\Lambda)$ is called tight if
$S_{g,\Lambda }  =  C \, \mathrm{Id}$ for some $C >0$. In this case
$\gamma = S_{g,\Lambda }\inv g = C\inv g$, and the
Gabor expansion 
$$
f= C\inv \sum_{\lambda \in \Lambda}
\<f, \pi(\lambda) \gamma\> \pi(\lambda)
g
$$
looks like an orthonormal expansion. 

Tight Gabor frames  with the constant $C=1$ can be constructed as
follows.   Let $\mathcal{G}(g,\Lambda)$ be a Gabor frame with frame
operator $S_{g,\Lambda}$. Due to \cite[Lemma 5.16]{Gro06}
$S^{-1/2}_{g,\Lambda} \mathcal{G}(g,\Lambda)$ is a tight Gabor frame
with constant $C=1$.  By  applying this procedure to a Gaussian
window, one sees that there exist tight Gabor frames $\cG (g,\Lambda
)$ with $g\in \cS (\rd )$. 

For future references we remark that if
$p\in (0,\infty ]$ and $f\in \mathcal S'(\rd)$,
then
\begin{align}\label{eq:normequivalence}
f\in M^p(\rd ) \quad \Leftrightarrow \quad
C_gf\in \ell ^p(\Lambda )
\quad \Leftrightarrow \quad
C_\gamma f\in \ell ^p(\Lambda ),
\end{align}
with norm equivalence 
$
\nm f{M^p} \asymp
\nm {C_gf}{\ell ^p}
\asymp
\nm {C_\gamma f}{\ell ^p}.
$

\subsection{Pseudodifferential Operators}

For a real-valued $d\times d$-matrix  $A \in \mathbb R^{d\times d}$,
and a symbol $\mathfrak{a} \in \mathcal{S}'(\mathbb{R}^{2d})$ the
pseudodifferential operator $Op_A(\mathfrak{a})$ is defined
by
\begin{align*}
	Op_A(\mathfrak{a}) u(x)
	= 
	\int \hspace{-2mm}
	\int \mathfrak{a}(x-A(x-y), \xi) f(y)
	e^{2\pi i(x-y)\cdot \xi}\, dy d\xi,
	\qquad u \in \mathcal{S}(\rd),
\end{align*} 
where the integrals should be interpreted in
distribution sense, if necessary.
If $A=0$, then $Op _A(\mathfrak{a})$ agrees with the
Kohn-Nirenberg or normal representation $a(x,D)$.
If instead $A=\frac 12 I$, where $I$ is the
$d\times d$ identity matrix, then $Op _A(\mathfrak{a})$
is the Weyl quantization $\mathfrak{a}^w$ of $\mathfrak{a}$.

\par

By \cite{Hoe83,To17B} for each symbol
$\mathfrak{a}_1 \in \mathcal{S}'(\mathbb{R}^{2d})$
and each $A_1, A_2  \in \mathbb R^{d\times d}$
there is a
unique $\mathfrak{a}_2 \in 
\mathcal{S}'(\mathbb{R}^{2d})$ such that $Op_{A_1}
(\mathfrak{a}_1)=Op_{A_2}(\mathfrak{a}_2)$, and that 
\begin{align}\label{eq:ChanceOfQuantization}
	Op_{A_1}(\mathfrak{a}_1)=Op_{A_2}(\mathfrak{a}_2) \Leftrightarrow \mathfrak{a}_2(x,\xi) = e^{2\pi it ((A_1-A_2)D_{\xi})
	\cdot D_x} \mathfrak{a}_1(x,\xi).
\end{align}
We refer to \cite{To17B} for the proof of
the following result.

\par

\begin{prop}\label{Prop:CalcTransfer}
Let $A,A_1,A_2\in \mathbb R^{d\times d}$ and
$p,q\in (0,\infty ]$. Then the following is true:
\begin{enumerate}
\item $e^{2\pi it (AD_{\xi})\cdot D_x}$
is an isomorphism  
on $M^{p,q}(\mathbb R^{2d})$;

\item $Op _{A_1}(M^{p,q}(\mathbb R^{2d}))
=
Op _{A_2}(M^{p,q}(\mathbb R^{2d}))$.
\end{enumerate}
\end{prop}

\par

\medspace

We can write a Weyl operator by means of  the Wigner distribution $W$ of $f,g \in L^2(\rd)$ which is defined by
\begin{align}\label{def:WignerDistribution}
	W(f,g)(x,\xi):= \int_{\rd} f(x+t/2)\overline{g(x-t/2)} e^{2\pi i \xi \cdot t}\, dt, \qquad x,\xi \in \rd.
\end{align}
Denote the inversion $\text{\u{g}}$ of $g\in \mathcal{S}'(\rd)$ by
$\text{\u{g}}(x):=g(-x)$ for all $x\in \rd$. Then 
\begin{align*}
W(f,g)(x,\xi)= 2^d e^{4\pi i x \cdot \xi} V_{\text{\u{g}}}f(2x,2\xi)
  \, ,
\end{align*}
and  the Wigner
distribution is just a slight modification  of the short-time Fourier
transform. 
Since  the short-time Fourier transform $V_{.}:\s \times \s \rightarrow \mathcal{S}(\mathbb{R}^{2d})$ we immediately get $W(f,g) \in \mathcal{S}(\mathbb{R}^{2d})$ for all $f,g \in \s$.

By means of the Wigner distribution the Weyl operator of a 
symbol $\mathfrak{a} \in \mathcal{S}'(\Rtdst)$ is given by the formula

\begin{align*}
	\< \mathfrak{a}^{w} f,g \>= \<\mathfrak{a}, W(g,f)\> \qquad \text{ for all } f,g\in \s.
\end{align*}
Pseudodifferential operators of Weyl form are continuous maps from $\s$ to $\mathcal{S}'(\rd)$, see
 \cite{To01} and \cite{Hoe83}. Moreover they are continuous as maps
 between certain modulation spaces, cf. \cite{Gro01} and
 \cite{GrHe99}.  
 
\begin{prop}\label{Prop:EmbeddingsPDO2}
 Let $0<p_0 \leq 1 $, $p,q \in [p_0, \infty]$
 and $\mathfrak{a} \in M^{\infty, p_0}(\rdd)$. Then
 $\mathfrak{a} ^w$ is bounded  on $M^{p,q}(\rd)$.
\end{prop}
 
As proved in \cite[Theorem 3.1]{To17}, this theorem also holds for
more general weighted modulation spaces.  

\begin{rem}\label{Rem:ModSpChirpInv}
By \cite{To01} it follows that $M^{p,q}(\rdd)$
is invariant under actions with chirps
$e^{i (AD_{\xi})\cdot D_x}$ with $A 
\in M(d \times d; \mathbb{C})$ for all $p,q \in 
(0,\infty]$. Hence all results concerning
Weyl operators of this paper also 
hold for operators of the form $Op_A(a)$.
\end{rem}

\par

Next we show that Gabor frame operators
with windows in $M^{p_0}$
are pseudo-differential operators with
symbols in $M^{\infty ,p_0}$.

\par

\begin{prop}\label{Prop:FrameOpPseudoDiff}
Let $p_0\in (0,1]$, $g_1,g_2\in M^{p_0}(\rd)$,
$A\in M(d \times d; \mathbb{R})$ be a $d\times d$-matrix, $\Lambda
\subseteq \rdd$ a lattice, 
and let $\mathfrak{a} \in \mathscr S'(\rdd )$
be such that $S_{g_1,g_2,\Lambda}=Op _A(\mathfrak{a})$.
Then $\mathfrak{a} \in M^{\infty ,p_0}(\rdd )$.
\end{prop}

\par

\begin{proof}
In view of Remark~\ref{Rem:ModSpChirpInv} we may assume that $A=0$. 
The Weyl symbol of
$S_{g_1,g_2,\Lambda}$ with $g_1=g_2$ was calculated in 
\cite[p.12]{AFK14}, and by 
similar arguments it follows that
the Kohn-Nirenberg symbol is 
\begin{align}
	\mathfrak{a} (x,\xi)
	&= \sum_{(l, \lambda ) \in \Lambda} g_1(x-l)
	\overline{ \widehat g_2(\xi-\lambda) }
	e^{2\pi i(x-l)\cdot (\lambda-\xi)}.
\end{align}

\par

In order to estimate the $M^{\infty ,p_0}$
norm of $\mathfrak{a}$ we choose the window
$\Psi \in \mathcal{S} (\rdd)\setminus \{0\}$ as
$$
\Psi(x,\xi):=\psi(x)
\overline{\widehat{\psi}(\xi)}
e^{-2\pi ix\cdot \xi}.
$$
By straightforward applications of
Fourier inversion formula, it follows that
\begin{align}
    (V_{\Psi}\mathfrak{a} )(z,w)
    =
    \sum_{l, \lambda \in \Lambda}
    e^{-2\pi i(y\cdot \xi+l\cdot \eta )} 
    (V_\psi g_1)(x-l,\xi +\eta -\lambda)
    \cdot 
    \overline{(V_\psi g_2)(x+y-l,\xi - \lambda)},
\label{Eq:STFTFrameOpPseudo}
\end{align}
%
(see e.g. \cite{GS07}) with $z=(x,\xi)\in \rd\times \rd $ and 
$w= ( \eta, y) \in \rd\times \rd
$.

\par

Now let $\Lambda ^2=\Lambda \times \Lambda$,
and $\Lambda ^4=\Lambda ^2\times \Lambda ^2$ and choose 
$\ep =\frac 1N$ with integer $N\ge 1$ large enough, 
such that
$$
\{ \psi (x-m)e^{2\pi ix\cdot \mu }\}
_{m,\mu \in \ep \Lambda}
\quad \text{and}\quad
\{ \Psi (z-w_1)e^{2\pi iz\cdot w_2 }\}
_{w_1,w_2\in \ep \Lambda ^2},
$$
are Gabor frames of $L^2(\rd )$ and $L^2(\rdd )$,
respectively.
Then
$$
\nm {g_j}{M^{p_0}}
\asymp
\nm {V_\psi g_j}{\ell ^{p_0}(\ep \Lambda ^2)}
\quad \text{and}\quad
\nm {\mathfrak{a}}{M^{\infty ,p_0}}
\asymp
\nm {V_\Psi \mathfrak{a}}
{\ell ^{\infty ,p_0}(\ep \Lambda ^4)},
$$
in view of a result in  \cite{GaS04}.

\par

Let $F_j=|V_\psi g_j|$.
If $w_1=(m,\mu )\in \ep \Lambda ^2$
and $w_2=(\nu,n )\in \ep \Lambda ^2$, then
\eqref{Eq:STFTFrameOpPseudo}
yields
\begin{align*}
|(V_{\Psi}\mathfrak{a} )(w_1,w_2)|^{p_0}
&\le
\sum_{l, \lambda \in \Lambda}
F_1(m-l,\mu +\nu -\lambda)^{p_0}
F_2(m+n-l,\mu -\lambda)^{p_0}
\\[1ex]
&=
(|\check F_1|^{p_0}*|F_2|^{p_0})(Tw_2),
\qquad Tw_2=(n,-\nu ).
\end{align*}
The right-hand side  does not depend on $w_1$, and therefore 
\begin{multline*}
\nm {\mathfrak{a}}{M^{\infty ,p_0}}
\asymp
\nm {V_\Psi \mathfrak{a}}
{\ell ^{\infty ,p_0}(\ep \Lambda ^4)}
\lesssim
\nm {{\check F}_1^{p_0}*F_2^{p_0}}
{\ell ^1(\ep \Lambda ^2)}^{\frac 1{p_0}}
\\[1ex]
\le
\big (
\nm {F_1^{p_0}}{\ell ^1(\ep \Lambda ^2)}
\nm {F_2^{p_0}}{\ell ^1(\ep \Lambda ^2}
\big )^{\frac 1{p_0}}
=
\nm {F_1}{\ell ^{p_0}(\ep \Lambda ^2)}
\nm {F_2}{\ell ^{p_0}(\ep \Lambda ^2)}
\asymp
\nm {g_1}{M^{p_0}}\nm {g_2}{M^{p_0}} \, .
\end{multline*}
Thus $\mathfrak{a} \in M^{\infty ,p_0}$. 
\end{proof}

\vspace{ 4 mm}

\subsection{Almost Diagonalization of Pseudodifferential Operators}

In this section we list several   characterizations of symbols in 
$\mathfrak{a}$ of $M^{\infty, p_0}(\rdd)$, $p_0 \in (0,\infty]$. The
following characterization of a symbol class by means of the almost
diagonalization of the associated \psdo\ was found in \cite[Theorem
3.2]{Gro06} for $p_0=1$ and subsequently generalized to the full range
of $p_0$ in \cite[Thm.~3.2]{BaCo22} (with almost the same proof).  
This characterization helps to deduce spectral properties of   \psdo s
from spectral properties of infinite matrices. 
We only formulate the unweighted case of \cite[Thm.~3.2]{BaCo22}, which is sufficient for this paper. 

\begin{tm}{(Almost Diagonalization).}\label{tm:AlmostDiag}
	We fix a non-zero $g\in \s\backslash\{0\}$ and a lattice $\Lambda \subseteq \Rtdst$ such that  $\mathcal{G}(g,\Lambda)$ is a Gabor frame for $L^2(\rd)$. Then for any $p_0\in (0,\infty]$ the following properties are equivalent:
	\begin{itemize}
		\item[i)] $\mathfrak{a} \in M^{\infty,p_0}(\Rtdst)$.
		\item[ii)] $\mathfrak{a} \in \mathcal{S}'(\Rtdst)$ and
                  there exists a function $H \in L^{p_0}(\Rtdst)$, in
                  fact $H\in W(C_b,L^{p_0})$,   such
                  that 
		\begin{align}
			|\langle \mathfrak{a}^w \pi (z)g, \pi (w)g 
			\rangle| \leq H(w-z) \qquad \forall w,z \in \Rtdst.
		\end{align}
		\item[iii)] $\mathfrak{a} \in \mathcal{S}'(\Rtdst)$ and there exists a sequence $h \in \ell^{p_0}(\Lambda)$ such that
\begin{align}
	|\langle \mathfrak{a}^w \pi(\rho)g, \pi(\lambda)g \rangle
	| \leq h(\lambda-\rho) \qquad \forall \lambda,\rho \in \Lambda.
\end{align}
\end{itemize}
\end{tm}

As an immediate application we obtain the following result.

\begin{cor}
	Under the hypotheses of Theorem \ref{tm:AlmostDiag} we assume that $T:\s \rightarrow \mathcal{S}'(\rd)$ is continuous and satisfies 
	\begin{align*}
		|\langle T \pi(\rho)g, \pi(\lambda)g \rangle| \leq h(\lambda-\rho) \qquad \forall \lambda,\rho \in \Ztdst.
	\end{align*}
	for some $h \in \ell^{p_0}(\Ztdst)$. Then $T=\mathfrak{a} ^w$ for some $\mathfrak{a}  \in M^{\infty, p_0}(\rdd)$.
\end{cor}

\begin{proof}
	Because of the continuity of $T:\s \rightarrow
        \mathcal{S}'(\rd)$  the Schwartz's kernel theorem asserts the existence of a symbol $\mathfrak{a}  \in \mathcal{S}'(\mathbb{R}^{2d})$ with $T=\mathfrak{a} ^w$.  An application of Theorem  \ref{tm:AlmostDiag} yields the claim. 
\end{proof}

\vspace{ 4 mm}

\subsection{Matrix Formulation} \label{subsection:MatrixFormulation}

We fix a lattice $\Lambda$ and a window $g \in \s
\backslash\{0\}$ such that $\mathcal{G}(g,\Lambda)$ is a  frame
with dual window $\gamma $ and  associated  Gabor expansion \eqref{GaborExpansionOfFunction} and
\eqref{GaborExpansionOfFunction2}. 

For the manipulations to be meaningful, 
we assume  of  a symbol $\mathfrak{a}  \in \mathcal{S}'(\rd)$ that the
Weyl operator $\mathfrak{a} ^w$ is bounded  on $M^p(\rd)$  for some $p
\in (0,\infty]$.  Just keep in mind, that due to 
Proposition \ref{Prop:EmbeddingsModulationSpaces} $M^p(\rd) 
\subseteq M^{\infty}(\rd)$. On account of the Gabor expansion 
\eqref{GaborExpansionOfFunction} we have $D_g C_{\gamma}f=f$ 
for all $f\in M^p(\rd) \subseteq M^{\infty}(\rd)$.  Together 
with the continuity of $\mathfrak{a} ^w$ on $M^p(\rd)$ we obtain for 
all $f\in M^p(\rd)$
\begin{align}\label{eq:CoefficientOperator}
	C_g(\mathfrak{a} ^wf)(\lambda) 
	=
	\langle \mathfrak{a} ^w f, \pi(\lambda) g \rangle 
	=
	\sum_{\mu \in \Lambda} \langle f, \pi(\mu)\gamma \rangle
	\langle \mathfrak{a} ^w \pi(\mu)g, \pi(\lambda)g \rangle
	\quad \text{for all } \lambda \in \Lambda.
\end{align}
We define the matrix $M(\mathfrak{a} )$ of $\mathfrak{a}^w$ with
respect to the frame $\cG (g,\Lambda )$ by its entries
\begin{align*}
	M(\mathfrak{a} )_{\lambda, \mu} = \langle \mathfrak{a} ^w
  \pi(\mu) g, \pi(\lambda)g \rangle \qquad \text{for all } \lambda,
  \mu \in \Lambda \, .
\end{align*}
We  can then  recast \eqref{eq:CoefficientOperator} as 
\begin{align}\label{eq:CoefficientOperator2}
	C_g(\mathfrak{a} ^wf)=  M(\mathfrak{a} ) C_{\gamma} f.
\end{align}
Likewise, Theorem~\ref{tm:AlmostDiag} now reads as follows.
\begin{cor}\label{cor:PropertyMa}
  Let $\mathfrak{a}\in \cS ' (\rdd )$. Then $\mathfrak{a}\in M^{\infty
    ,p_0}(\rdd )$, \fif\ $M(\mathfrak{a}) \in \cC ^{p_0}$. 
\end{cor}

 The matrix $P$ associated to the identity operator has  the entries
 \begin{align} \label{xy}
	P_{\lambda, \mu} = \langle  \pi(\mu) g, \pi(\lambda)g \rangle,
   \qquad \text{for all } \lambda, \mu \in \Lambda.  
\end{align}
To simplify the analysis of $P$, we assume from now on that $\cG (g,
\Lambda )$ is a tight frame with $S_{g,\Lambda } = \mathrm{I}$.  As
mentioned already, tight frames with a window in $\mathcal{S}(\rd )$
always exist. 
The matrix $P$ has the following properties. 
\begin{lemma}\label{lemma:PropertiesProjection}
Assume  that $\cG (g,
\Lambda )$ is a tight frame with $g\in \mathcal{S}(\rd )$. Let $0<p_0 \leq p \leq \infty$ and $\mathfrak{a}  \in \mathcal{S}'(\Rtdst)$ be such that the associated pseudodifferential operator $\mathfrak{a} ^w$ is bounded on $M^p$. Then
\begin{itemize}
	\item[(i)] $P$ is a projection from $\ell^p(\Lambda)$ to range
          $C_g(M^p)$, i.e., $Pc=c \in \ell ^p(\Lambda )$ \fif\  
          there  exists $f\in M^p$ 
such that $c_\lambda =\scal f{\pi(\lambda)g} = (C_gf)_ \lambda$ for  $\lambda \in \Lambda$.

	\item[(ii)]  $PM(\mathfrak{a} )=M(\mathfrak{a} )$ and $M(\mathfrak{a} )P=M(\mathfrak{a} )$.
        \item[(iii)]   	$P \in \mathcal{C}^{p}(\Lambda)$.
	\item[(iv)] For $\mathfrak{a}  \in M^{\infty,p_0}(\Rtdst)$ we have $M(\mathfrak{a} )+I-P \in \mathcal{C}^{p_0}(\Lambda)$.
\end{itemize}
\end{lemma}
\begin{proof}
(i)   For all $\lambda, \nu \in \Lambda$ the assumption that
$S_{g,\Lambda } = \mathrm{Id}$ and  
 \eqref{Eq:FrameOp} imply that 
\begin{align*}
	(P^2)_{\lambda, \nu} 
	&= \Big\langle \sum_{\mu \in \Lambda} \langle \pi(\nu) g, \pi(\mu) g \rangle \pi(\mu) g, \pi(\lambda) g \Big\rangle
	= \langle S_{g,\Lambda }\pi (\nu ) g,  \pi (\lambda ) g
  \rangle = \langle \pi(\nu) g, \pi(\lambda) g\rangle \\
	&=P_{\lambda, \nu}.
\end{align*}
Consequently  $P^2=P$ and $P$ is a projection. 
Next  let $c \in \ell^p(\Lambda)$ with $Pc=c$. With
$f=\sum_{\mu \in \Lambda} c_{\mu} \pi(\mu) g$ we have  for $\lambda \in \Lambda$, 
\begin{align*}
	c_{\lambda} 
	= (Pc)_{\lambda}
	=\sum_{\mu \in \Lambda} \langle \pi(\mu) g, \pi(\lambda) g \rangle c_{\mu}
	=\langle f,\pi(\lambda) g \rangle
	=(C_g f)_{\lambda}\, ,
\end{align*}
whence $c = C_gf$. Then the norm-equivalence
\eqref{eq:normequivalence} yields 
$f \in M^p(\rd )$.

Conversely, assume that  $c=C_g f$ for some $f \in M^p(\rd)$. Then  we get due to \eqref{GaborExpansionOfFunction}
\begin{align*}
	(Pc)_{\lambda} 
	= \langle \sum_{\mu \in \Lambda} \langle f, \pi(\mu) g \rangle  \pi(\mu) g , \pi(\lambda) g \rangle
	= \langle S_{g,\Lambda } f , \pi (\lambda ) g \rangle = \langle f, \pi(\lambda) g \rangle
	=c_{\lambda}, \quad \lambda \in \Lambda \, ,
\end{align*}
and (i) holds.

(ii)  is proved similarly by straight forward calculations using 
\eqref{Eq:FrameOp}. 

(iii) If $g\in \cS (\rd )$, then also $V_gg\in \cS (\rdd )$, e.g., by
\cite[Thm.~11.2.5]{Gro01}. In particular, this implies that  for every
$N\geq 0$
$$
|V_gg(z) |  \lesssim (1+|z|)^{-N} \, , 
$$
and thus
$$
|P_{ \lambda,\mu }| =  |\langle \pi (\mu )g,\pi (\lambda )g\rangle | =
|V_gg(\lambda -\mu )| \lesssim (1+|\lambda -\mu |)^{-N}\, .
$$
By choosing $N$ large enough, we see that $H(z) = (1+|z|)^{-N}$ is in
$W(C_b,L^{p_0})$ and consequently $P\in \cC ^{p_0}(\Lambda )$. 

(iv) Since all matrices $P,\mathrm{I}$ and $M(\mathfrak{a})$ are in $\cC
^{p_0}$, their sum 
   $M(\mathfrak{a})+ \mathrm{Id} + P$ is also in $\cC
   ^{p_0}$.
\end{proof}

\vspace{3mm}

\subsection{Spectral invariance.}

We  have already seen before that it is possible to relate  a pseudodifferential operator 
$\mathfrak{a}  ^w$ to  an infinite matrix $M(\mathfrak{a} )$.
It turns out that there is  a connection between the invertibility of $\mathfrak{a}  ^w$ and $M(\mathfrak{a} )$. 
\begin{lemma}\label{lemma:equivalenceInvertibility}
Assume  that $\cG (g,
\Lambda )$ is a tight frame with $g\in \mathcal{S}(\rd )$. 	Let $0<p\leq \infty$ and $\mathfrak{a}  \in
        \mathcal{S}'(\rdd)$ be such that  the associated Weyl operator
        $\mathfrak{a}  ^w$ is bounded  on $M^p$. Then
        $\mathfrak{a}  ^w$ is invertible  on $M^p$,  
if and only if
\begin{itemize}
\item[i)]
$
\nm {M(\mathfrak{a}  )Pc}p \gtrsim \nm {Pc}p\quad \text{ for all } c\in \ell ^p(\Lambda)
$ and
\item[ii)] for every $c_0\in P\ell ^p(\Lambda)$, there is a $c\in P\ell ^p(\Lambda)$ such that
$M(\mathfrak{a}  )Pc=Pc_0$.
\end{itemize}
Here the projection $P$ is defined as in \eqref{xy}. 
\end{lemma}
\begin{proof}
Let $\mathfrak{a}  ^w$ be invertible on $M^p(\rd)$. 
Using \eqref{eq:CoefficientOperator2}, we obtain
\begin{align*}
	\|M(\mathfrak{a} ) C_g f\|_{p}
	= \|C_g (\mathfrak{a}  ^w f)\|_{p} 
	\asymp \|\mathfrak{a}  ^w f\|_{M^p}
	\gtrsim \|f\|_{M^p}
	\asymp \|C_g f\|_{p} \, ,
\end{align*}
which is (i). 
To prove (ii),  let $c_0 \in P \ell^p(\Lambda)$ be arbitrary. 
Then there exists  $h \in M^p(\rd)$ with $C_g h =c_0=Pc_0$ by Lemma~\ref{lemma:PropertiesProjection}. 
Since $\mathfrak{a}  ^w$ is bijective on $M^p(\rd)$, there is a $f \in M^p(\rd)$ such that 
\begin{align*}
	\mathfrak{a}  ^w f= h.
\end{align*}
Then we obtain for $c=C_g f$ due to Lemma \ref{lemma:PropertiesProjection} that $Pc=c$ and that
\begin{align*}
	M(\mathfrak{a} )Pc
	=M(\mathfrak{a} )c
	=M(\mathfrak{a} )C_g f
	=C_g (\mathfrak{a}  ^w f)
	= C_g h
	=c_0 = Pc_0 .
\end{align*}
This implies (ii).

Conversely  assume that (i) and (ii) hold. Using
\eqref{eq:CoefficientOperator2} (with $\gamma = g$)  and (i),   we obtain for all $f \in M^p(\rd)$
\begin{align*}
	\|\mathfrak{a}  ^w f\|_{M^p}
	\asymp \|C_g (\mathfrak{a}  ^w f)\|_p
	= \|M(\mathfrak{a} ) C_g f\|_p
	\gtrsim \|C_g f\|_{p}
	\asymp \|f\|_{M^p}.
\end{align*}
Hence $\mathfrak{a}  ^w $ is one-to-one on $M^p(\rd )$. 
To prove that $\mathfrak{a}  ^w $ is surjective we choose an arbitrary
$h \in M^p(\rd)$ and let $c_0:=C_g h \in
\ell^p(\Lambda)$. Then  
$c_0=Pc_0 \in P \ell^p(\Lambda)$. 

By  assumption (ii) there is a $c \in P \ell^p(\Lambda)$ such that 
\begin{equation}\label{eq:M_SigmaIstCg}
	M(\mathfrak{a} ) Pc = Pc_{0} =c_0=C_g h. 
\end{equation}
Since the image of $P$ is $C_g(M^p)$, there is a $f\in M^p(\rd)$ with
\begin{align*}
	Pc=C_g f.
\end{align*}
A combination of \eqref{GaborExpansionOfFunction}, 
\eqref{eq:CoefficientOperator2}
and \eqref{eq:M_SigmaIstCg}  yields 
\begin{align*}
	\mathfrak{a}  ^w f
	= D_{g}C_g(\mathfrak{a}  ^w f) 
	= D_{g} M(\mathfrak{a} ) Pc
	= D_{g} C_g h
	=h.
\end{align*}
This implies that $\mathfrak{a}  ^w$ maps  onto $M^p(\rd )$  and hence invertible on $M^p(\rd)$.
\end{proof}
In the previous lemma we proved the equivalence of the invertibility of $\mathfrak{a}  ^w$ on 
$M^p(\rd)$
and of $M(\mathfrak{a} )$ on 
$P \ell^p(\Lambda)$.  Since $\mathrm{ker}\, P \neq \{0\}$ and
$M(\mathfrak{a}) = M(\mathfrak{a})P$, $M(\mathfrak{a}) $ cannot be
invertible on the whole space $\ell ^p(\Lambda )$.
In the literature this problem is usually overcome by using the
pseudo-inverse of $M(\mathfrak{a})$ and holomorphic functional
calculus. Here  we use a new trick, which may be of independent
interest.  Consider the matrix $A =
M(\mathfrak{a}) + \mathrm{Id} - P$. 
We can then  use the spectral invariance result for infinite 
convolution-dominated matrices of  Theorem \ref{Thm:SpectralInvarianceMatrix_new}
to derive a 
spectral invariance result for pseudodifferential operators 
on \modsp s, cf. Theorem \ref{tm:SpectralInvarianceModSpaces}.
\begin{proof}[Proof of Theorem \ref{tm:SpectralInvarianceModSpaces}]
Let $p \in [p_0,\infty]$ be the index  for which $\mathfrak{a}
^w$ is invertible on $M^p(\rd)$ and let $A=M(\mathfrak{a}) +
\mathrm{Id} - P$, where $P$ is the projection defined in
\eqref{xy}.
First we check the assumptions of Theorem 
\ref{Thm:SpectralInvarianceMatrix_new} and prove that 
\begin{align}\label{eq:AInvertible}
 A=M(\mathfrak{a}  )+I-P \qquad \text{ is invertible on  } \ell^p(\Lambda).
\end{align}
Assume that  
$Ac=0$ for some $c \in \ell^p(\Lambda)$. Then by i) and ii) of Lemma \ref{lemma:PropertiesProjection}
 \begin{align*}
 	0=Ac=\left( M(\mathfrak{a}  )+I-P \right) \left( Pc+(I-P)c \right)=M(\mathfrak{a}  )Pc+(I-P)c.
 \end{align*} 
 If we apply $P$ respectively $I-P$ to the previous equality and use Lemma \ref{lemma:PropertiesProjection} again, we obtain 
 \begin{align*}
 M(\mathfrak{a}  )Pc=0 \qquad \text{and} \qquad (I-P)c=0. 
 \end{align*}
 Since $\mathfrak{a}  ^w$ is invertible on $M^p$ we obtain by  Lemma \ref{lemma:equivalenceInvertibility} and the previous estimate
 \begin{align*}
 	\|Pc\|_p \lesssim \|M(\mathfrak{a} )Pc\|_p=0,
 \end{align*}
 and consequently $Pc=0$. Hence $c=Pc+(I-P)c=0$ which shows  that $A$
 is one-to-one.
 
 To show the surjectivity of $A$, we let $c_0 \in \ell^p(\Lambda)$ be arbitrary.
 Since  $\mathfrak{a}  ^w$ is invertible on $M^p$,   Lemma
 \ref{lemma:equivalenceInvertibility} yields  the existence of $c \in P \ell^p(\Lambda)$ with 
 \begin{align}\label{eq:MPc}
 	M(\mathfrak{a} )Pc=Pc_0.
 \end{align}
 Then 
 $\widetilde{c}=Pc+(I-P)c_0 \in \ell^p(\Lambda)$ and by  Lemma \ref{lemma:PropertiesProjection} and \eqref{eq:MPc}
 \begin{align*}
 	 A \widetilde{c}=
 	 M(\mathfrak{a}  )\widetilde c+(I-P)\widetilde c
=
M(\mathfrak{a}  )Pc+(I-P)c_0=Pc_0+(I-P)c_0 =c_0.
 \end{align*}
 Thus $A$ is onto on $\ell^p(\Lambda)$ and   therefore invertible on
 $\ell ^p(\Lambda )$. 
 Due to Lemma \ref{lemma:PropertiesProjection} we have $A \in 
 \mathcal{C}^{p_0}(\Lambda)$. Since \eqref{eq:AInvertible} also holds we can apply Theorem 
 \ref{Thm:SpectralInvarianceMatrix_new}  and get the invertibility of  $A$  on $\ell^q(\Lambda)$ for all $q \in [p_0,\infty)$. \\
 
 Next we show, that 
 \begin{align}\label{eq:I}
 	M(\mathfrak{a} ) \text{ is invertible on } P\ell^q(\Lambda) \text{ for all } q\in [p_0, \infty). 
 \end{align}
 Let $q \in  [p_0,\infty)$ be arbitrary.  Since $I-P \equiv 0$ on $P \ell^q(\Lambda)$, 
 the injectivity of $A$ implies that $M(\mathfrak{a} )$ is one-to-one on $P\ell^q(\Lambda)$. 
 For an arbitrary $c_0 \in P \ell^q(\Lambda) \subseteq \ell	^q(\Lambda)$, 
 there is some $c \in \ell^q(\Lambda)$ with $Ac=c_0$, 
 since $A$ is onto on $\ell^q(\Lambda)$. Since $(I-P)M(\mathfrak{a}) =
 0$  by Lemma \ref{lemma:PropertiesProjection}, we obtain, after
 applying  $I-P$ to $Ac=c_0$, that  $0=(I-P)c_0=(I-P)Ac=(I-P)c$. Then $c \in P\ell^q(\Lambda)$ and
 \begin{align*}
 	M(\mathfrak{a} )c=M(\mathfrak{a} )c + (I-P)c =Ac=c_0.
 \end{align*}
 Hence $M(\mathfrak{a} )$ is onto  $P \ell^q(\Lambda)$ and \eqref{eq:I} holds. 
By   Lemma \ref{lemma:equivalenceInvertibility} $\mathfrak{a}$ is
invertible on $M^q$. 
\end{proof}

We now prove Theorem~\ref{tm:SpectralInvarianceNew} and obtain more
refined information about the inverse $(\mathfrak{a}^w)^{-1}$. 

\begin{proof}[Proof of Theorem \ref{tm:SpectralInvarianceNew}]
By Theorem \ref{tm:SpectralInvarianceModSpaces} we get the invertibility of $\mathfrak{a} ^w$ on $M^2(\rd)=L^2(\rd)$. 
 	By  \cite[Theorem 4.6]{Gro06} and the embedding
        $M^{\infty,p_0}(\mathbb{R}^{2d}) \subseteq
        M^{\infty,1}(\mathbb{R}^{2d})$,  there is a symbol $\mathfrak{b} \in M^{\infty, 1}(\Rtdst)$ with 
 	$\mathfrak{b} ^w = (\mathfrak{a} ^w)\inv $. 
We consider the associated matrices $M(\mathfrak{a})$ and
$M(\mathfrak{b})$ with respect to a tight Gabor frame $\cG (g,\Lambda
)$ with $g\in \cS (\rd )$ and again denote by $P$ the projection with
entries $P_{\lambda \mu } = \langle \pi (\mu )g,\pi (\lambda )g
\rangle$. 
On account of Lemma \ref{lemma:PropertiesProjection} we get for  all $c \in \text{ran }C_g$ the existence of an $f\in M^2(\rd)$ with 
$c=C_g f=Pc$.
Then for 
	 all $c = C_g f = Pc\in \text{ran }C_g$ we  obtain  using \eqref{eq:CoefficientOperator2}
 	\begin{align*}
M(\mathfrak{b})M(\mathfrak{a})c = M(\mathfrak{b})M(\mathfrak{a}) C_g
          f= M(\mathfrak{b}) C_g(\mathfrak{a}^w f)=
          C_g(\mathfrak{b}^{w} \mathfrak{a}^w  f) =C_gf = c \, .
 	\end{align*}
If $Pc = 0$, then $M(\mathfrak{a})c = M(\mathfrak{a})Pc = 0$,
consequently on $\ell ^2(\Lambda )$ we have  
 $$
        M(\mathfrak{b})M(\mathfrak{a})   = P \, .
 $$
It follows that
$$
\big(M(\mathfrak{b}) + \mathrm{Id}-P\big) \big(M(\mathfrak{a}) +
\mathrm{Id} -P\big) = \big(M(\mathfrak{b}) + \mathrm{Id}-P\big) \, A =  \mathrm{Id} \, .
$$
This means that $B = M(\mathfrak{b}) + \mathrm{Id}-P$ is the inverse
of the invertible matrix $A $ (since
the inverse is unique). 
Since $A \in \cC ^{p_0}$ by Lemma~\ref{lemma:PropertiesProjection},
Theorem~\ref{tm:SpectralInvariance}  implies that also $B \in \cC ^{p_0}$.

Consequently $M(\mathfrak{b})\in \cC ^{p_0}$. Now the characterization
of Corollary~\ref{cor:PropertyMa} implies that $\mathfrak{b} \in M^{\infty ,p_0}(\rdd
)$, as claimed. 
\end{proof}

By using Theorem~\ref{tm:SpectralInvarianceNew} we can now deduce the
invertibility of $\mathfrak{a}$ on more general modulation spaces which generalizes Theorem \ref{tm:SpectralInvarianceModSpaces}.

\begin{tm}[Spectral invariance on Modulation Spaces] \label{thm:SprectralInvariance} 
	If  $\mathfrak{a} \in M^{\infty,p_0 }(\rdd )$ for $p_0\in (0,1]$ and  $\mathfrak{a}   ^w$ is invertible on 
	%
  $M^{p}(\rd)$ for some $p \in [p_0,\infty]$,
	then $\mathfrak{a}^w $ is also invertible on  $M^{r,q}(\rd)$
        for all  $r,q \in [p_0,\infty)$.
\end{tm}
\begin{proof}
	On account of Theorem \ref{tm:SpectralInvarianceNew} there is
        a $\mathfrak{b}  \in M^{\infty,p_0 }(\rdd )$ with
        $\mathfrak{b}  ^w = (\mathfrak{a} ^w)\inv $ on $M^p(\rd)$. By
        Proposition \ref{Prop:EmbeddingsPDO2}  $\mathfrak{b} ^{w}$ is
        bounded  on $M^{p,q}(\rd)$. Since
        $\mathfrak{b} ^{w}\mathfrak{a}^w  =\mathfrak{a}^w \mathfrak{b}
        ^{w}=I$ on $\s \subseteq M^p(\rd)$ we obtain the invertibility
        of $\mathfrak{a}^w $ on $M^{p,q}(\rd)$ by the density of $\s$
        in  $M^{p,q}(\rd)$. 
\end{proof}

This theorem is an extension of
\cite[Corollary 4.7]{Gro06} from
the case $p_0=1$ to $p_0<1$. By using
the arguments of ~\cite{GR08},
one can formulate the corollary for an even more 
general class of \modsp s. 

\par

\begin{rem}
Let $A\in \mathbb R^{d\times d}$.
Proposition \ref{Prop:CalcTransfer}
implies  that the conclusions in Theorems 
\ref{tm:SpectralInvarianceNew},
\ref{tm:SpectralInvarianceModSpaces}
and Theorem
\ref{thm:SprectralInvariance}
remain true with $Op _A(\mathfrak a)$
and $Op _A(\mathfrak b)$ in place of
$\mathfrak a ^w$ and $\mathfrak b^w$,
respectively, at each occurrence.
\end{rem}

\par

As an application of Theorem  \ref{thm:SprectralInvariance} we show
the following property of the canonical dual window of an Gabor
frame. 
\begin{tm}
    Let $0<p\leq 1$, $\Lambda \subseteq \Rtdst$ be
    a lattice and $g\in M^p(\rd)$ be such that
    $\mathcal{G}(g,\Lambda)$ is a Gabor frame
    for $L^2(\rd)$.
    Then the canonical dual window 
    \begin{align}\label{tm:DualWindow}
    	\gamma= S_{g, \Lambda}^{-1} g \in M^p(\rd).
    \end{align}
\end{tm}
%
\begin{proof}
We denote the Kohn-Nirenberg symbol  of the frame operator $S_{g, \Lambda}$ by $\mathfrak{a}$.
By  Proposition \ref{Prop:FrameOpPseudoDiff} we have
\begin{align*}
	\mathfrak{a} \in M^{\infty, p }(\rdd).
\end{align*}
Since Theorem \ref{thm:SprectralInvariance} also holds for
pseudodifferential operators in the Kohn-Nirenberg quantization, 
$S_{g, \Lambda}$ is invertible on $M^p(\rd)$. 
Therefore \eqref{tm:DualWindow} holds. 
\end{proof}

\appendix

\section{Proofs of some preparatory results}
\label{App:A}

In this appendix we prove some preparatory
results from Section \ref{sec:Prelim} and  \ref{sec:SpectralInvariance}.

\begin{proof}[Proof of Lemma
\ref{ConnWienAmalgSpaceSequenceSpace2}]
Let $\alpha>0$ be chosen such that 
$[0,\alpha]^d\subseteq B_{1/2}(0)$ and 
$(B_{1/2}(\alpha k))_{k \in \zd}$ covers
$\rd$. Then
\begin{align*}
\|H\|_{\ell^{p_0}(\Lambda)}^{p_0}
\leq
\sum_{k \in \zd} \sum_{\Lambda \cap B_{1/2}
(\alpha k )} |H(\lambda)|^{p_0}
\lesssim
\sum_{k \in \zd} \rel(\Lambda) 
\int_{[0,\alpha]^d}
\sup_{\lambda \in B_{1/2}(\alpha k )} 
|H(\lambda)|^{p_0} dy.
\end{align*}
Since $B_{1/2}(\alpha k) \subseteq B_1(\alpha k + 
y)$ for each $y \in [0,\alpha]^d$ the integral
just becomes larger, if we take the supremum over 
all $\lambda \in  B_1(\alpha k + y)$ instead of 
$\lambda \in B_{1/2}(\alpha k)$. A substitution 
then provides 
\begin{align*}
\|H\|_{\ell^{p_0}(\Lambda)}^{p_0} 
\lesssim
\sum_{k \in \zd} \rel(\Lambda) 
\int_{[0,\alpha]^d+\alpha k}\sup_{\lambda \in 
B_{1}(x)} |H(\lambda)|^{p_0} \, dx
=
\rel(\Lambda) \|H\|_{W(C_b,L^{p_0})}^{p_0}.
\end{align*}
\end{proof}

\begin{proof}[Proof of Proposition 
\ref{Prop:ConvDomOpCont}]
For $\rho _0\in \Pi$ and
$\lambda _0\in \Lambda$ fixed let
$$
s_{\Lambda ,p_0}(\rho _0)^{p_0}
:=
\sum _{\lambda \in \Lambda}H(\lambda -\rho _0)^{p_0}
\quad \text{and}\quad
s_{\Pi ,p_0}(\lambda _0)^{p_0}
:=
\sum _{\rho \in \Pi}H(\lambda _0-\rho )^{p_0}.
$$
Then it follows by straightforward
estimates that
$$
s_{\Lambda ,p_0}(\rho )
\le
C_\Lambda
\nm H{W(C_b,L^{p_0})},
\quad
s_{\Pi ,p_0}(\lambda )
\le
C_\Pi \nm H{W(C_b,L^{p_0})},
\quad
\lambda \in \Lambda ,\ \rho \in \Pi,
$$
where $C_\Lambda =C_0\operatorname{rel}(\Lambda )^{\frac 1{p_0}}$,
for some constant $C_0>0$ which only depends on
$d$.

\par
 
Suppose $q\ge 1$. Then H{\"o}lder's inequality
together with the facts that $s_{\Lambda ,p_0}$
and $s_{\Pi ,p_0}$ decrease with $p_0$ give
\begin{align*}
\|Ab\|_{\ell^{q}(\Lambda)}^q
&\leq
\sum_{\lambda \in \Lambda}
\left( \sum_{\rho \in \Pi}
(H(\lambda-\rho)^{\frac 1q}|b_{\rho}|)
H(\lambda-\rho)^{\frac 1{q'}}\right)^{q}
\\[1ex]
&\le
\sum_{\lambda \in \Lambda}
\left( \sum_{\rho \in \Pi}
H(\lambda-\rho)|b_{\rho}|^q\right )
\left( \sum_{\rho \in \Pi}
H(\lambda-\rho)\right)^{\frac q{q'}}
=
\sum_{\lambda \in \Lambda}
\left( \sum_{\rho \in \Pi}
H(\lambda-\rho)|b_{\rho}|^q\right )
s_{\Pi ,1}(\lambda )^{\frac q{q'}}
\\[1ex]
&\le
(C_\Pi \nm H{W(C_b,L^{p_0})})^{\frac q{q'}}
\sum_{\rho \in \Pi}
\left(\sum_{\lambda \in \Lambda} H(\lambda-\rho)\right )|b_{\rho}|^q
\\[1ex]
&=
(C_\Pi \nm H{W(C_b,L^{p_0})})^{\frac q{q'}}
\sum_{\rho \in \Pi}
s_{\Lambda ,1}(\rho )|b_{\rho}|^q
\le
C_\Pi ^{\frac q{q'}}\nm H{W(C_b,L^{p_0})}^q
C_\Lambda \nm b{\ell ^q(\Pi )}^q,
\end{align*}
giving the assertion when $q\ge 1$.

\par

If instead $p_0 \le q\le 1$, then
\begin{align*}
\|Ab\|_{\ell^{q}(\Lambda)}^q
&\leq
\sum_{\lambda \in \Lambda}
\left( \sum_{\rho \in \Pi}
H(\lambda-\rho)|b_{\rho}|
\right)^{q}
\le
\sum_{\rho \in \Pi}
\left( 
\sum_{\lambda \in \Lambda}
H(\lambda-\rho)^q\right)|b_{\rho}|^{q}
\\[1ex]
&\le
\sum_{\rho \in \Pi}
s_{\Lambda ,q}(\rho)^q
|b_{\rho}|^{q}
\le
C_\Lambda ^q\nm H{W(C_b,L^{p_0})}^q\sum_{\rho \in \Pi}
|b_{\rho}|^{q}
=
(C_\Lambda \nm H{W(C_b,L^{p_0})}\nm b{\ell ^q(\Pi )})^q,
\end{align*}
giving the result for $q\le 1$.
\end{proof}

\begin{proof}[Proof of Lemma \ref{lemma:Normequivalence1} for $q<1$]
Since $\supp(\varphi) \subseteq B_2(0)$, we get
\begin{align}
\label{eq_sj_covnum}
\eta  := \sup_{\varepsilon>0} \sup_{x \in \rd} \#\set{k \in
  \mathbb{Z}^d}{\varphi^\varepsilon_k(x) \not= 0} = \sup_{\varepsilon>0} \sup_{x
  \in \rd} \# \{k \in 
  \mathbb{Z}^d \cap B_2(\varepsilon x)\}  <\infty.
\end{align}

So we  obtain the following bound  for all $x \in \rd$: 
\begin{align*}
1 \leq \sum_{k\in \zd} \varphi^\varepsilon_k(x)^q
= \sum_{k\in \zd: \varphi^\varepsilon_k(x) \not= 0}
\varphi^\varepsilon_k(x) ^q
\leq \eta  \, \sup_{k \in \zd} \varphi^\varepsilon_k(x)^q \leq \eta .
\end{align*}
Therefore, for all $x \in \rd$, 
\begin{align}
\label{eq_sj_partq}
\frac{1}{\eta} \leq \sup _{k\in \zd} \varphi^\varepsilon_k(x)^q   \leq
 \sum_{k\in \zd} \varphi^\varepsilon_k(x)^q  
\leq \eta \, .
\end{align}

If   $a \in \ell^q(\Pi )$, then 
\begin{align*}
\frac{1}{\eta } \sum_{\rho \in \Pi} \abs{a_\rho}^q
\leq \sum_{\rho \in \Pi} \sum_{k\in \zd} \varphi^\varepsilon_k(\rho)^q \abs{a_\rho}^q
\leq \eta \sum_{\rho \in \Pi} \abs{a_\rho}^q,
\end{align*}
which implies the claim
with constants independent of $\varepsilon$ and $q$.
\end{proof}

 \bibliographystyle{abbrv}
 \bibliography{general,new}

\end{document}

%% file: macros.tex
\newcommand{\up}{uncertainty principle}
\newcommand{\tfa}{time-frequency analysis}

\newcommand{\tf}{time-frequency}

\newcommand{\fif}{if and only if}

\newcommand{\modsp}{modulation space}
\newcommand{\psdo}{pseudodifferential operator}

\newtheorem{tm}{Theorem}[section]    
\newtheorem{lemma}[tm]{Lemma}
\newtheorem{prop}[tm]{Proposition}
\newtheorem{cor}[tm]{Corollary}
\newtheorem{definition}[tm]{Definition}

\numberwithin{equation}{section} 

\theoremstyle{definition}

\newtheorem{rem}{Remark}[section]

\newcommand{\beqa}{\begin{eqnarray*}}
\newcommand{\eeqa}{\end{eqnarray*}}

\DeclareMathOperator*{\supp}{supp}

\newcommand{\field}[1]{\mathbb{#1}}
\newcommand{\bR}{\field{R}}        
\newcommand{\bZ}{\field{Z}}        
\newcommand{\bC}{\field{C}}        
\def\cS{\mathcal{ S}}

\def\cB{\mathcal{ B}}
\def\cG{\mathcal{ G}}

\def\cA{\mathcal{ A}}

\def\cC{\mathcal{ C}}

\def\rd{\bR^d}

\def\rdd{{\bR^{2d}}}
\def\zdd{{\bZ^{2d}}}

\def\lrd{L^2(\rd)}

\def\zd{\bZ^d}

\def\intrd{\int_{\rd}}

\def\<{\left<}
\def\>{\right>}

\def\inv{^{-1}}

\def\mv1{M_v^1}

\newcommand{\abs}[1]{\lvert#1\rvert}
\newcommand{\norm}[1]{\lVert#1\rVert}

